\documentclass[10pt, a4paper,reqno]{amsart}
\usepackage{latexsym}
\usepackage{amsmath}
\usepackage[dvips]{graphicx}%
\usepackage{amsfonts}%
\usepackage{amssymb}
\usepackage{mathtools}
\usepackage{mathrsfs}

\usepackage{color}
\definecolor{purple}{rgb}{0.9,0.01,0.3}

\def\R{\mathbb R}
\def\N{\mathbb N}
\def\l{\lambda}

\def\cal{\mathcal}
\def\ov\overline
\def\Hper{\widetilde H^1_\#(\Gamma_h)}

\def\H{{\cal H}}
\def\K{{\cal K}}

\def\F{{\cal F}}
\def\hd{{\rm hd}}

\def\I{\mathcal{I}}

\def\cc{\subset\subset}
\def\d{{\rm d}\,}

\def\a{\alpha}
\def\b{\beta}

\def\g{\gamma}
\def\de{\delta}
\def\e{\varepsilon}

\def\s{\sigma}

\newcommand{\medint}{-\kern -,375cm\int}
\newcommand{\medintinrigo}{-\kern -,315cm\int}

\def\Om{\Omega}

\def\pa{\partial}
\def\Id{{\rm Id}\,}
\def\loc{{\rm loc}}
\def\Div{{\rm div}}
\newcommand{\dist}{{\rm dist}}
\newcommand{\E}{\mathcal E}
\def\S{\Sigma}

\def\tr{{\rm tr}\,}

\numberwithin{equation}{section}
\newtheorem{theorem}{Theorem}[section]

\newtheorem{corollary}[theorem]{Corollary}

\newtheorem{lemma}[theorem]{Lemma}

\newtheorem{proposition}[theorem]{Proposition}

\theoremstyle{definition}
\begingroup

\newtheorem{remark}[theorem]{Remark}

\endgroup

\begin{document}

\title[]{Existence and almost everywhere regularity \\ of isoperimetric clusters for fractional perimeters}

\author{M. Colombo}
\address{Institute for Theoretical Studies, ETH Z\"urich, Clausiusstrasse 47, CH-8092 Z\"urich, Switzerland
\\
Institut f\"ur Mathematik, Universitaet Z\"urich, Winterthurerstrasse 190, CH-8057 Z\"urich,
Switzerland}
\email{maria.colombo@math.uzh.ch}

\author{F. Maggi}
\address{Department of Mathematics, The University of Texas at Austin,  2515 Speedway Stop C1200, Austin, Texas 78712-1202, USA}
\email{maggi@math.utexas.edu}

\maketitle

\begin{abstract}
{\rm The existence of minimizers in the fractional isoperimetric problem with multiple volume constraints is proved, together with a partial regularity result.}
\end{abstract}

%%%%%%%%%%%%%%%%%%%%%%%%%%%%%%%%%%%%%%%%%%%%%%%%%%%%%
%%%%%%%%%%%%%%%%%%%%%%%%%%%%%%%%%%%%%%%%%%%%%%%%%%%%%
%%
%%      INTRODUCTION
%%
%%%%%%%%%%%%%%%%%%%%%%%%%%%%%%%%%%%%%%%%%%%%%%%%%%%%%
%%%%%%%%%%%%%%%%%%%%%%%%%%%%%%%%%%%%%%%%%%%%%%%%%%%%%

\section{Introduction} The goal of this paper is establishing basic existence and partial regularity results for the fractional isoperimetric problem with multiple volume constraints. If $E\subset\R^n$, $n\ge 2$, and $s\in(0,1)$, then the {\it fractional perimeter of order $s$ of $E$} is defined as
\begin{equation}
  \label{fractional perimeter}
  P_s(E)=\int_{\R^n}\,w_E(x)\,dx=\int_E\,dx\int_{E^c}\frac{dy}{|x-y|^{n+s}}\,.
\end{equation}
The kernel $z\mapsto|z|^{-n-s}$ is not integrable near the origin, and the potential
\[
w_E(x):=1_E(x)\,\int_{E^c}\frac{dy}{|x-y|^{n+s}}\qquad x\in\R^n
\]
explodes like $\dist(x,\pa E)^{-s}$ as $x\in E$ approaches $\pa E$. Since $t^{-s}$ is integrable near $0$, by decomposing the integral of $w_E$ on a small layer around $\pa E$ as the integral along the normal rays $t\mapsto p-t\,\nu_E(p)$, $p\in\pa E$, then we see that $P_s(E)$, at leading order, is measuring the perimeter $P(E)=\H^{n-1}(\pa E)$ of $E$. This idea is made precise by the fact that, as $s\to 1^-$, $(1-s)\,P_s(E)\to c(n)\,P(E)$ for every set of finite perimeter $E$, see \cite{Bourgain01anotherlook,davilaBV}, and $(1-s)\,P_s\to P$ in the sense of $\Gamma$-convergence \cite{ambrosiodephilippismartinazzi}.

The last few years has seen a great effort by many authors towards the understanding of geometric variational problems in the fractional setting. This line of research has been initiated in \cite{caffaroquesavin} with the regularity theory for the fractional Plateau's problem (see \cite{savinvaldinoci,figallivaldinociLipschitz,barrerafigallivaldinoci,caputoguillen,caffarellivaldinoci} for further developments in this direction), while fractional isoperimetric problems have been the subject of \cite{frankliebseiringer,knupfermuratovI,knupfermuratovII,dicastronovagaruffinivaldinoci,F2M3}. Examples of singular fractional minimal boundaries (boundaries with vanishing fractional mean curvature) are found in \cite{daviladelpinowei,daviladelpinowei2}. Boundaries with constant fractional mean curvature have also been investigated in some detail \cite{daviladelpiniodipierrovaldinoci,cabrefallweth1,cabrefallweth2,ciraolofigallimagginovaga} and their study illustrates how nonlocality brings into play both complications (need for new arguments, for example when in the local case one exploits some direct ODE argument) and simplifications (because of additional rigidities): compare, for example, the stability results from \cite{ciraolomaggi} with those in \cite{ciraolofigallimagginovaga}.

Our goal is starting the study, in the fractional setting, of another classical geometric variational problem, namely the isoperimetric problem with multiple volume constraints. Given $N\in \N$, a $N$-cluster (or simply a cluster) $\E$ is a family $\E=\{\E(h)\}_{h=1}^N$ of disjoint Borel subsets of $\R^n$. The sets $\E(h)$, $h=1,..., N$, are called the chambers of $\E$, while $\E(0)=\R^n\setminus\bigcup_{h=1}^N\E(h)$ is called the exterior chamber of $\E$. When $|\E(h)|<\infty$ for $h=1,...,N$, then the volume vector $m(\E)$ of $\E$ is defined as
\[
m(\E)=(|\E(1)|,...,|\E(N)|)\in\R^N
\]
while the fractional $s$-perimeter of $\E$ is given by
\begin{equation}
  \label{fractional perimeter cluster}
  P_s(\E)=\frac12\sum_{h=0}^NP_s(\E(h))\,.
\end{equation}
Given $m\in\R^N_+$ (that is, $m_h>0$ for $h=1,...,N$), we consider the following isoperimetric problem
\begin{equation}
  \label{fractional isoperimetric problem}
  \inf\Big\{P_s(\E):m(\E)=m\Big\}\,.
\end{equation}
Every minimizer in \eqref{fractional isoperimetric problem} is called an isoperimetric cluster. The following theorem is our main result.

\begin{theorem}\label{thm main}
For every $m\in\R^N_+$ there exists an isoperimetric cluster $\E$ with $m\,(\E)=m$. If we set
\begin{equation}
  \label{boundary of E cluster}
  \pa\E=\Big\{x\in\R^n:\mbox{$\exists h=1,...,N$ such that $0<|\E(h)\cap B_r(x)|<|B_r(x)|$ $\forall r>0$}\Big\}
\end{equation}
then $\pa\E$ is bounded and there exists a closed set $\S(\E)\subset\pa\E$ such that $\H^{n-2}(\S(\E))=0$ if $n\ge 3$, $\S(\E)$ is discrete if $n=2$, and $\pa\E\setminus\S(\E)$ is a $C^{1,\alpha}$-hypersurface in $\R^n$ for some $\a\in(0,1)$.
\end{theorem}

Let us review the theory of isoperimetric clusters when the classical perimeter, not the fractional one, is minimized. This theory has been initiated by Almgren \cite{Almgren76} with the proof of the analogous result to Theorem \ref{thm main}, namely an existence and $C^{1,\a}$-regularity theorem out of a closed singular set of Hausdorff dimension $n-1$. When $n=2$ the only singular minimal cone consists of three half-lines meeting at 120 degrees at a common end-point, so that, by a standard dimension reduction argument, the singular set has Hausdorff dimension at most $n-2$, and is discrete when $n=2$. (This estimate is of course sharp.) Taylor \cite{taylor76} has proved that, if $n=3$, then the only singular cones are obtained either by the union of three half-planes meeting at 120 degrees along a common line, or as cones spanned by the edges of regular tetrahedra over their barycenters; and that, moreover, $\pa\E$ is locally $C^{1,\a}$-diffeomorphic to its tangent cone at every point, including singular ones. The regular part $\pa\E\setminus\S(\E)$ has constant mean curvature and is real analytic, in dimension $n=3$ up to the singular set \cite{nitsche,kindspruck}. Regularity of and near the singular set in dimension $n\ge 4$ seems still to be an open problem.

Explicit examples of isoperimetric clusters are known just in a few cases. In the case of two chambers ($N=2$) the only isoperimetric clusters are double-bubbles, whose boundaries consist of three $(n-1)$-dimensional spherical caps meeting at $120$-degrees along a $(n-2)$-dimensional sphere; see \cite{small_doublebubble} in dimension $n=2$, \cite{HutchMorganRosRitore} ($n=3$) and \cite{Reichardt,ReichardtColl} ($n\ge 4$). In the case of three chambers ($N=3$) one can define a candidate isoperimetric cluster, the so-called triple bubble, enclosing three given volumes. When $n=2$, the minimality of this triple bubble was proved in \cite{wichi}. Another important isoperimetric problem is partitioning a flat torus into chambers of equal volumes. In the case $n=2$ this problem has been solved in \cite{hales}, where the minimality of hexagonal honeycomb partitions is proved. Global stability inequalities for planar double bubbles and for hexagonal honeycombs have been obtained in \cite{CiLeMaFUGLEDE} and \cite{carocciamaggi}, together with quantitative descriptions of minimizers in the presence of a small potential term.

The present paper naturally opens two kind of questions, which are actually closely related: first, understanding singularities of fractional isoperimetric clusters and, second, characterizing fractional isoperimetric clusters in some basic cases. Thinking about the arguments used to achieve these goals in the local theory, the extension to the fractional case is necessarily going to require the introduction of new arguments and ideas.

One may speculate that the fractional theory may be helpful in advancing the local theory: on the one hand, depending on the question under study, the rigidity of nonlocal perimeters may end up bringing in some simplifications with respect to the local case; on the other hand, information in the classical setting can be recovered from the fractional case in the limit $s\to 1^-$. In any case, at present, the viability of this idea has not been really tested on specific examples.

The paper is divided into two sections. In section \ref{section existence} we prove the existence part of Theorem \ref{thm main} by adapting to the fractional setting Almgren's original proof (as presented in \cite[Part IV]{maggiBOOK}). In section \ref{section regularity} we prove the partial regularity assertion in Theorem \ref{thm main}. Similarly to what done in \cite{caffaroquesavin} for fractional perimeter minimizing boundaries, we exploit an extension problem to obtain a monotonicity formula, showing that nearby most points of the boundary only two chambers of the isoperimetric cluster are present. When this happens we can show that the two neighboring chambers locally almost-minimize fractional perimeter, and we can thus apply the main result in \cite{caputoguillen} to prove their $C^{1,\a}$-regularity.

\bigskip

\noindent  {\bf Acknowledgment:} Work partially supported by NSF DMS Grant 1265910 and DMS-FRG Grant 1361122.
%
%
%
%{\color{blue}
%The proof
%
%is for many aspects analogous to ... but for others different.
%
%Filosofia generale: non si scrive la ODE per l'infiltration lemma oppure la mappa che spinge il cluster per fargli crescere il volume, ma esistono per ragioni astratte...
%
%Controllo delle costanti si perde col volume-fixing.%in modo che si veda come esplodono per $s \to 1$.
%Le scriviamo per gli altri lemmi.
%
%fare il tripunto non e' banale, o descrivere la doppia bolla.
%
%Il problema e' vettoriale, quindi classificare i coni e' un problema.
%
%Sarebbe analitico se si potesse applicare ... ma quello e' scritto solo per cose che in una palla sono min, non $(\Lambda, r_0)$-min (def...)
%}

\section{Existence theorem}\label{section existence} The goal of this section is proving the existence of isoperimetric clusters of any given volume. Precisely, given $m\in\R^N_+$ we discuss the existence of isoperimetric sets of volume $m$, that is, minimizers in
\begin{equation}
\label{isoperimetric problem fractional existence}
\g=\inf \big\{ P_s(\E): \E \mbox{ is an $N$-cluster in $\R^n$ with }m (\E) = m \big\}\,.
\end{equation}
Every minimizing sequence $\{\E_k\}_{k\in\N}$ is compact in $L^1_{{\rm loc}}(\R^n)$ (see section \ref{section notation} for the terminology used here and in the sequel) and fractional perimeters are trivially lower semicontinuous with respect to this convergence, so that the only difficulty in showing the existence of minimizers is the possibility that minimizing sequences do not converge in $L^1(\R^n)$ (loss of volume at infinity). Almgren's strategy to fix this problem \cite{Almgren76} (which predates by a decade the formalization of this kind of argument in the theory of concentration-compactness!) consists in {\it nucleating}, {\it truncating}, {\it volume-fixing} and {\it translating} a given minimizing sequence. The nucleation step consists in decomposing the cluster $\E_k$ into finitely many ``chunks'' which contain most of the volume. These chunks are defined by intersecting the chambers of $\E_k$ with a finite collection of balls of equal radii, each chunk having bounded diameter and possibly diverging from the others. In the truncation step  the chambers of $\E_k$ are ``chopped'' by a slight enlargement of the nucleating balls in such a way that the perimeter is decreased by an amount which is proportional to the volume left out. The volume is then restored by slight deformations of the clusters. By these operations one has obtained a new minimizing sequence, localized into finitely many regions of bounded diameter. In the classical case, where local perimeter is minimized, one can finally translate these nuclei so to obtain a new minimizing sequence entirely contained in a bounded region. In the nonlocal case one cannot freely translate disconnected parts of the cluster without changing in a complex way its fractional perimeter. However, in section \ref{section boundedness} we show that once a sequence of clusters have bounded fractional perimeter and is localized into finitely many (possibly diverging) regions of bounded diameter, then the sequence is actually bounded (see Lemma \ref{lemma:boundedness}). In sections \ref{section volumefixing}, \ref{section truncation}, \ref{section nucleation} we take care, respectively, of the volume-fixing, truncation and nucleation steps of the argument, highlighting the differences brought in Almgren's argument by the nonlocality of fractional perimeters. Finally, in section \ref{section proof of existence} we combine these tools to prove the existence of isoperimetric sets.

\subsection{Notation and terminology}\label{section notation} Given disjoint Borel sets $E,F\subset\R^n$ and $s\in(0,1)$, we define the fractional interaction energy of order $s$ between $E$ and $F$ by setting
$$
I_s(E,F) = \int_E\int_{E^c}\frac{dx\, dy}{|x-y|^{n+s}}\,.
$$
The fractional $s$-perimeter is then given by $P_s(E)=I_s(E,E^c)$, see \eqref{fractional perimeter}. The $s$-perimeter of $E\subset\R^n$ relative to an open set $\Omega\subset\R^n$ is defined by the formula
\[
P_s(E;\Omega)=I_s(E\cap \Omega, E^c\cap \Omega) + I_s(E\cap \Omega, E^c\cap \Omega^c) + I_s(E\cap \Omega^c, E^c\cap \Omega)\,.
\]
The motivation for this definition lies in the fact that if $P_s(E)$ and $P_s(F)$ are both finite and $E\cap\Om^c=F\cap\Om^c$, then $P_s(E)-P_s(F)=P_s(E;\Om)-P_s(F;\Om)$.

A $N$-cluster, or simply a cluster, is a family $\E=\{\E(h)\}_{h=1}^N$ of disjoint sets, called the chambers of $\E$. The set  $\E(0)=\R^n\setminus\bigcup_{h=1}^N\E(h)$ is called the exterior chamber of $\E$. The volume of $\E$ is the vector $m(\E)=(|\E(1)|,...,|\E(N)|)$. The {\it relative distance between the $N$-clusters $\E$ and $\E'$ in $\Om\subseteq \R^n$} is defined by
$$
d_\Om(\E, \E') = \sum_{h=0}^N |\Om \cap (\E(h) \Delta \E'(h))|\,.
$$
The {\it relative $s$-perimeter $P_s(\E;\Om)$ of the cluster $\E$ in $\Om$} is defined as
$$
P_s(\E; \Omega) = \frac{1}{2} \sum_{i=1}^M P_s(\E(i);\Omega)\,,
$$
so that $P_s(\E)=P_s(\E;\R^n)$, see \eqref{fractional perimeter cluster}. We say that a sequence $\{\E_k\}_{k\in\N}$ of $N$-clusters converges in $L^1_{{\rm loc}}(\Om)$ to a $N$-cluster $\E$ if $1_{\E_k(h)}\to 1_{\E(h)}$ in $L^1_{\rm loc}(\Om)$ for every $h=1,...,N$. If $\sup_{k\in\N}P_s(\E_k;\Om)<\infty$, then one can find a subsequence of $\{\E_k\}_{k\in\N}$ which admits an $L^1_{{\rm loc}}(\Om)$ limit. Finally, the boundary of a Borel set $E\subset\R^n$ is defined as
\begin{equation}
\label{boundary of E set}
\pa E=\Big\{x\in\R^n:0<|E\cap B_r(x)|<|B_r(x)|\qquad\forall r>0\Big\}\,.
\end{equation}
In this way \eqref{boundary of E cluster} is equivalent to
\[
\pa\E=\bigcup_{h=1}^N\pa\E(h)\,.
\]

\subsection{A boundedness criterion}\label{section boundedness} The following lemma exploits the rigidity of fractional perimeters to show that a cluster consisting of finitely many pieces localized in different bounded regions has actually bounded diameter.

\begin{lemma}\label{lemma:boundedness}
Let $\{\E_k\}_{k\in\N}$ be a minimizing sequence for \eqref{isoperimetric problem fractional existence}. Let us assume that there exist positive constants $R$ and $c$ and, for every $k\in\N$, finitely many points $\{x_k(i)\}_{i=1,...,L(k)}$, with the property that
\begin{eqnarray}
\label{hp:confinement}
\E_k(h) \subseteq \bigcup_{i=1}^{L(k)} B_R(x_{k}(i))\,,&&\qquad \forall  k\in \N\,, h=1,...,N\,,
\\
\label{hp:finite-number}
\sup_{k\in\N}L(k) <\infty\,,&&
\\
\label{eqn:un-poco-di-massa}
\sum_{h=1}^{N} |\E_k(h)\cap B_R(x_{k}(i))| \geq c\,,&&\qquad \forall i=1,..., L(k)\,, k\in \N\,.
\end{eqnarray}
Then there exists $R_0>0$ and a subsequence (not relabelled) %$\{ x_k\}_{k\in \N}$
such that $\E_{k}(h) \subseteq B_{R_0}(x_{k}(1))$ for every $h = 1,...,N$ and for every $k\in\N$.
\end{lemma}

Before proving the lemma, we recall that the $s$-perimeter is subadditive, and more precisely for every couple of disjoint measurable sets $E, F\subseteq \R^n$ we have
 \begin{equation}
 \label{eqn:p-s-subadd}
  P_s(E)+ P_s(F) - \frac{2|E| \, |F|}{{\rm dist}(E, F)^{n+s}} \leq P_s(E \cup F)
  \leq P_s(E)+ P_s(F) - \frac{2|E| \, |F|}{{\rm diam}(E \cup F)^{n+s}}\,.
 \end{equation}
 Indeed, we have that $P_s(E \cup F) =   P_s(E)+ P_s(F) - 2\,I_s(E,F)$ and
$$\frac{|E| \, |F|}{{\rm diam}(E \cup F)^{n+s}} \leq I_s(E,F) \leq \frac{|E| \, |F|}{{\rm dist}(E, F)^{n+s}}.$$
This observation will be applied to estimate the perimeter of a sequence of clusters with a finite number of ``components'' which are moving away from each other.

\begin{lemma}\label{lemma:additive-per-limit}
Let $E_k \subseteq \R^n$ be a sequence of measurable sets such that
$$
E_k \subseteq \bigcup_{i=1}^L B_R(x_k(i)) \qquad \forall k\in \N\,,
$$
where $R>0$, $L\in\N$ and, for each $i=1,...,L$, $\{x_k(i)\}_{k\in \N}$ are sequences of points such that
\begin{equation}
\label{hp:centers-away-lemmino}
\lim_{k\to \infty}\inf_{1\le i< j\le L} |x_k(i)-x_k(j)| = \infty\,.
\end{equation}
Then
\begin{equation}
\label{eqn:limit-e-k-lemmino}
\lim_{k\to \infty} \Big| P_s(E_k)-  \sum_{i=1}^L P_s\big(E_k \cap B_R(x_k(i))\big) \Big| = 0\,.
\end{equation}
 \end{lemma}
 \begin{proof} The inequality
 $$
 P_s(E_k)\leq \sum_{i=1}^L P_s\big(E_k \cap B_R(x_k(i))\big)
 $$
 follows from the subadditivity of the $s$-perimeter. Moreover, by induction over \eqref{eqn:p-s-subadd}, given $L$ sets $F_1$,..., $F_L$ whose mutual distances are bigger than $D>0$, one has
\begin{equation}
\begin{split}
P_s\Big( \bigcup_{i=1}^L F_i \Big) \geq \sum_{i=1}^L P_s(F_i) -
\frac{ 2L^2 \max_{i=1,...,L} |F_i|^2}{D^{n+s}}\,.
\end{split}
\end{equation}
Given $k \in \N$, we apply this inequality to the sets $F_i=E_k \cap B_R(x_k(i)))$. Since in this case we have $|F_i|\le |B_R|$ and $D\ge\min_{i\neq j }|x_k(i)-x_k(j)| - 2R$, we obtain
$$
P_s(E_k)\geq \sum_{i=1}^L P_s\big(E_k \cap B_R(x_k(i))\big) - \frac{ 2L^2 |B_R|^2}{\min_{i\neq j }(|x_k(i)-x_k(j)| - 2R)^{n+s}}\,.
$$
By \eqref{hp:centers-away-lemmino} we obtain \eqref{eqn:limit-e-k-lemmino}.
 \end{proof}

\begin{proof}[Proof of Lemma~\ref{lemma:boundedness}.]
We argue by contradiction, assuming that there exists a minimizing sequence $\{\E_k\}_{k\in\N}$ in \eqref{isoperimetric problem fractional existence} such that \eqref{hp:confinement}, \eqref{hp:finite-number}, and \eqref{eqn:un-poco-di-massa} hold, but with
\[
\lim_{k\to\infty}\max_{1\le h\le N}{\rm diam}\,(\E_k(h))=+\infty\,.
\]
Up to extracting a subsequence, we may assume that $L(k)=L_0$ independent on $k$.

\medskip

\noindent {\it Step one}: We claim that there exist $L \in \{ 2,.., L_0\}$ and  $S \geq R$ such that, up to extracting a subsequence in $k$ and up to reordering the set $\{1,...,L_0\}$, we have
\begin{eqnarray}
\label{hp:centers-away}
\lim_{k\to \infty} |x_k(i)-x_k(j)| = \infty\,,&&\qquad\forall i,j \in \{1,...,L\}, i\neq j,
\\
\label{hp:confinement-S}
\E_k(h) \subseteq \bigcup_{i=1}^L  B_S(x_k(i))\,,&&\qquad\forall k\in \N, \, h=1,...,N\,,
\\
\label{eqn:quanto-di-massa-S}
\sum_{h=1}^{N} |\E_k(h)\cap  B_S(x_k(i))| \geq c\,,&& \qquad \forall i=1,..., L,
\end{eqnarray}
where the constant $c$ is the one appearing in \eqref{eqn:un-poco-di-massa}. Indeed, up to extracting subsequences, we may assume that for every $i,j\in \{1,...,L_0\}$ there exists $S(i,j)=\lim_{k\to \infty} |x_k(i)- x_k(j)| \in [0,\infty]$. We then say that $\{x_k(i)\}_{k\in \N}$ and  $\{x_k(j)\}_{k\in \N}$ are {\it asymptotically close} if $S(i,j) <\infty$, and introduce an equivalence relation $\sim$ on $\{ 1,..., L_0\}$ so that $i \sim j$ if and only if $\{x_k(i)\}_{k\in \N}$ and  $\{x_k(j)\}_{k\in \N}$ are asymptotically close. Up to reordering $\{ 1,..., L_0\}$, we may assume that $L$ is such that $\{1,...,L\}$ contains exactly one representative of each equivalence class. Hence, \eqref{hp:centers-away} follows by the fact that representatives of different classes cannot be asymptotically close.
Finally, by taking $S:= \sup_{i,j\in \{1,...,L\}, i \sim j} S(i,j) +R$, we clearly have $B_R(x_{k}(i)) \subseteq B_S(x_{k}(j))$ for every $i,j\in \{1,...,L\}$ with $i \sim j$, so that \eqref{hp:confinement} implies \eqref{hp:confinement-S}. Finally, \eqref{eqn:quanto-di-massa-S} follows from \eqref{eqn:un-poco-di-massa} since $ B_R(x_{k}(i)) \subseteq   B_S(x_k(i))$.

\medskip

\noindent {\it Step two}: Up to further extracting subsequences and reordering indices, we may assume that
 %the couples $1$ and $2$ are the closest among all couples in $\{1,.., L\}$, namely
%%$$ \lim_{k\to \infty} |x_k(1)- x_k(2)|\leq \lim_{k\to \infty} |x_k(i)- x_k(j)| \qquad \mbox{for every }i,j\in \{1,...,L\}, i\neq j$$
%and even
$$
|x_k(1)-x_k(2)| \leq |x_k(i)-x_k(j)|\,,\qquad\forall k\in\N\,, i,j\in \{1,...,L\}\,, i\neq j\,.
$$
and that
\[
\lim_{k\to \infty} P_s(\E_k(h))\,\,\mbox{exists}\quad \forall h=1,...,L\,.
\]
Moreover, up to a translation and a rotation, we may assume that
\[
x_k(1)= 0\,,\qquad\frac{x_k(2)}{|x_k(2)|}=e_1\,.
\]
We now define a new sequence $\{\E'_k\}_{k\in\N}$ so that $\E_k'$ coincides with $\E_k$ in the balls  $B_S(x_k(i))$ with $i\neq 2$, whereas the part of $\E_k$ inside $B_S(x_{k}(2))$ is translated at distance $3S$ from $x_k(1)=0$: more precisely,
% {\color{blue} qui ci potrebbe andare una figura, con quello piu' vicino dei fuggitivi che viene ritraslato vicino a quello che resta al finito}: in other words,
for every $h=1,..., N$ we set
\begin{equation*}
\begin{split}
\E_k'(h) =&% \big( \E_k(h) \cap B(x_{k}(1), S) \big)\cup
  \Big( (\E_k(h) \cap B_S(x_{k}(2))) + (3S - |x_k(2)|)e_1 \Big)\cup
\bigcup_{i\neq 2 } \big( \E_k(h) \cap  B_S(x_k(i)) \big)\,.
%\\ &\cup  \big( \E_k(h) \cap B(0, S) \big).
\end{split}
\end{equation*}
By Lemma~\ref{lemma:additive-per-limit} applied to each chamber of $\E_k$ and to $\E_k(0)^c$  we have
\begin{equation*}
\begin{split}
2\gamma &= 2\lim_{k\to \infty} P_s(\E_k)
=\lim_{k\to \infty} \Big(P_s(\E_k(0)^c)+ \sum_{h=1}^N P_s(\E_k(h))\Big)
\\
&= \lim_{k\to \infty}\Big[ \sum_{i=1}^L  P_s\big(\E_k(0)^c \cap  B_S(x_k(i))\big) +
\sum_{h=1}^N \sum_{i=1}^L  P_s\big(\E_k(h) \cap  B_S(x_k(i))\big) \Big].
\end{split}
\end{equation*}
We can use the same argument on the chambers of $\E_k'$ which are contained in the balls $\{B_{4S}(0),B_S(x_k(i)):3\le i\le L\}$, in order to obtain
\begin{equation*}
%\label{eqn:limit-e-k-primo}
\begin{split}
2\,\limsup_{k\to \infty} P_s(\E'_k)&=\limsup_{k\to \infty} \Big(P_s(\E_k'(0)^c)+ \sum_{h=1}^N P_s(\E_k'(h))\Big)
\\
&=\limsup_{k\to \infty} \Big[ P_s\big(\E'_k(0)^c \cap B_{4S}\big) + \sum_{i=3}^L  P_s\big(\E_k(0)^c \cap  B_S(x_k(i))\big)
\\
& \hspace{4em}+\sum_{h=1}^N P_s\big(\E'_k(h) \cap B_{4S}\big) + \sum_{h=1}^N\sum_{i=3}^L  P_s\big(\E_k(h) \cap  B_S(x_k(i))\big) \Big]\,.
\end{split}
\end{equation*}
By combining these identities we get
\begin{equation}
\label{eqn:limit-e-k}
\begin{split}
2\,\gamma-2\,\limsup_{k\to \infty} P_s(\E'_k) =&  \limsup_{k\to \infty} \Big[ -P_s\big(\E'_k(0)^c \cap B_{4S}\big) + \sum_{i=1,2}  P_s\big(\E_k(0)^c \cap B_S(x_k(i))\big)
\\
& \hspace{4em}-\sum_{h=1}^N P_s\big(\E'_k(h) \cap B_{4S}\big) + \sum_{i=1,2} \sum_{h=1}^N P_s\big(\E_k(h) \cap  B_S(x_k(i))\big)\Big]\,.
\end{split}
\end{equation}
By the subadditivity of the $s$-perimeter, for every $k \in \N$ and $h=1,...,N$ one has
\begin{equation}
\label{eqn:non-male}
P_s\big(\E'_k(h) \cap B_{4S}\big) \leq P_s\big(\E_k(h) \cap B_S) +P_s(\E_k(h) \cap B_S(x_k(2))\big)\,.
\end{equation}
At the same time, for every $k \in \N$,
\begin{equation}
\label{eqn:guadagno}
P_s\big(\E'_k(0)^c \cap B_{4S}\big) \leq P_s\big(\E_k(0)^c \cap B_S\big) +P_s\big(\E_k(0)^c \cap B_S(x_k(2))\big)- \frac{2c^2}{(8S)^{n+s}}\,.
\end{equation}
To prove \eqref{eqn:guadagno}, we exploit the upper bound in \eqref{eqn:p-s-subadd} with $E= \E'_k(0)^c \cap B_S$ and $F= \E'_k(0)^c \cap B_S(3Se_1)$. Since $E \cup F \subseteq B_{4S}$ and $|E|, |F| \geq c$ by \eqref{eqn:un-poco-di-massa}, we find that
\begin{equation*}
\begin{split}
P_s\big(\E'_k(0)^c \cap B_{4S}\big) &= P_s\big((\E'_k(0)^c \cap B_S) \cup (\E'_k(0)^c \cap B_S(3Se_1))\big)
\\&\leq P_s\big(\E'_k(0)^c \cap B_S\big)+ P_s\big(\E'_k(0)^c \cap B_S(3Se_1)\big) - \frac{2c^2}{(8S)^{n+s}}.
\end{split}
\end{equation*}
Since $\E'_k(0)^c \cap B_S(3Se_1)$ is a translation of $\E_k(0)^c \cap B_S(x_k(2))$, we have prove \eqref{eqn:guadagno}. By combining  \eqref{eqn:non-male} and \eqref{eqn:guadagno} with \eqref{eqn:limit-e-k}, and taking into account that  each $\E_k'$ is a competitor in \eqref{isoperimetric problem fractional existence}, we finally find a contradiction, namely
\[
\g\le\limsup_{k\to \infty} P_s(\E'_k) \leq \gamma - \frac{c^2}{(8S)^{n+s}}\,.\qedhere
\]
\end{proof}

\subsection{Volume-fixing variations}\label{section volumefixing}
 In studying isoperimetric problems with multiple volume constraints one needs to use local diffeomorphic deformations to adjust volumes of competitors. (Scaling is not useful here, as it can just be used to fix the volume of a chamber per time.) This basic technique is found in Almgren's work \cite[VI-10,11,12]{Almgren76}. Here we follow the presentation of \cite[Sections 29.5-29.6]{maggiBOOK}, and discuss the adaptations needed to work in the fractional setting. Given a reference $N$-cluster $\E$, our goal is proving that for every cluster $\E'$ which is sufficiently $L^1$-close to $\E$ and for every volume $m'$ sufficiently close to $m(\E')$ there exists a deformation of $\E'$ with volume $m'$ and perimeter which has increased, at most, proportionally to the small quantity $|m'-m(\E')|$; see Proposition \ref{prop:vol-fix} below.

 The first step to achieve this is proving that, in any ball where the two chambers $\E(i)$ and $\E(j)$ are present, one can build a compactly supported vector field whose flow increases the volume of $\E(i)$ with speed $1$, decreases the volume of $\E(j)$ with speed $-1$, and leaves the volumes of the other chambers infinitesimally unchanged. In the local case this is done in a geometrically explicit way by exploiting the notion of reduced boundary to push $\E(i)$ along its (measure-theoretic) outer unit normal, compare with \cite[Section 29.5]{maggiBOOK}. In the fractional case we are not dealing with sets of finite perimeter, and we thus resort to a more abstract approach, which in fact simplifies the construction. In the following we set
 $$
 V=\big\{\bold a \in \R^{N+1}: \bold a (0)+ ... +\bold a (N) =0\big\}\,.
 $$
%
%  The lemma will be applied in small neighborhoods of a point of $\partial \E(i) \cap \partial \E(j)$. In the classical setting of perimeter-minimizing clusters, the following lemma would be proved only in a neighborhood of points of $\partial^* \E(i) \cap \partial^* \E(j)$, by building a vector field which approximates the outer normal of $\E(i)$ at $x_0$.
\begin{lemma}\label{lemma:campo}
 If $\E$ is an $N$-cluster in $\R^n$, $0\le i < j\le N$, and $z\in\pa\E(i)\cap\pa\E(j)$, then for every $R>0$ there exists a vector field $T_{ij} \in C^\infty_c(B_R(z);\R^n)$ such that
$$ \int_{\E(i)} \Div (T_{ij}) \, dx = 1 = - \int_{\E(j)} \Div (T_{ij}) \, dx, \qquad
 \int_{\E(h)} \Div (T_{ij}) \, dx =0 \qquad \forall h \neq i,j.
$$
\end{lemma}

\begin{proof} {\it Step one}: Given $R>0$ and $z\in\R^n$, let $H\subset\{0,...,N\}$ be such that $h\in H$ if and only if $0<|\E(h) \cap B_R(z)|<B_R(z)$.   Let us consider the linear operator $\mathcal L: C^\infty_c(B_R(z); \R^n) \to \R^{N+1}$ defined by
$$
\mathcal L (T) = \Big( \int_{\E(0)} \Div (T) \, dx , \; ... , \; \int_{\E(N)} \Div (T) \, dx \Big)\,,
$$
and consider the linear spaces
\[
I=\Big\{\mathcal L(T):T\in C^\infty_c(B_R(z); \R^n)\Big\}\qquad V'=\big\{\bold a\in V:\bold a(h)=0\quad\forall h\not\in H\big\}\,.
\]
We claim that $I=V'$. Trivially, $I\subset V'$. Since $I$ is the intersection of all the hyperplanes that contain it, it is enough to show that if $J$ is an hyperplane in $\R^{N+1}$ which contains $I$, then $V'\subset J$. Indeed, let $\{\l_h\}_{h=0}^N$ be such that $\bold a\in J$ if and only if $\sum_{h=0}^N\l_h\bold a(h)=0$. The condition $I\subset J$ implies that
$$
0 = \sum_{h\in H} \lambda_h  \int_{\E(h)} \Div (T) \, dx = \int_{\R^n} \Big( \sum_{h\in H} \lambda_h 1_{\E(h)} \Big) \Div (T) \, dx\,,\qquad\forall T \in C^\infty_ c(B_R(z);\R^n)\,,
$$
so that $\sum_{h\in H} \lambda_h 1_{\E(h)}$ is constant in $B_R(z)$. As the chambers $\E(h)$ are disjoint, this means that there exists $\l\in\R$ such that $\l_h=\l$ for every $h\in H$, and thus $V'\subset J$ holds.

\medskip

\noindent {\it Step two}: Now let $z\in\pa\E(i)\cap\pa\E(j)$ for some $0\le i < j\le N$, and given $R>0$ let $H\subset\{0,...,N\}$ be defined as in step one, so that $\{i,j\}\subset H$. Since $I=V'$ and the equations $\bold a(i)=1$, $\bold a(j)=-1$, $\bold a(h)=0$ for $h\ne i,j$ define an element $\bold a\in V'$, we conclude the existence of $T_{ij}\in C^\infty_c(B_R(z);\R^n)$ with the required properties.
\end{proof}

The subsequent step is checking that the flows generated by the vector-fields $T_{ij}$ found in the previous lemma have the required properties. We notice that the constant $C_0$ below depends also on $\|T \|_{C^1}$ (and therefore on our particular cluster), so the dependence on $s$ is not explicit here.

\begin{lemma}[Infinitesimal volume exchange between two chambers]\label{lemma:infinit-vol-fix}
 Let $s\in(0,1)$ and $\E$ be an $N$-cluster in $\R^n$. If $0 \le h<k\le N$, $z\in \partial \E(h) \cap \partial \E(k)$, and $r,\delta>0$, then there exist positive constants $\e_1$, $\e_2$, $C_0$ depending only on $\E, z, r, \delta$,
and a family of diffeomorphisms $\{ f_t\}_{|t|\leq \e_1}$ such that
\begin{equation}
\label{eqn:supp-f-t}
\big\{ x\in \R^n: x \neq f_t(x) \big\} \subset\subset B_r(z), \qquad \forall |t| \le \e_1,
\end{equation}
which satisfies the following properties:
\begin{enumerate}
\item[(i)] if $\E'$ is a cluster, $d(\E, \E') < \e_2$ (in particular, if $\E'=\E$), and $|t| <\e_1$, then
$$
\bigg| \frac{d}{dt} \Big| f_t\big(\E'(h)\big) \cap B_r(z) \Big| -1 \bigg| <\delta, \qquad \bigg| \frac{d}{dt} \Big| f_t\big(\E'(k)\big) \cap B_r(z) \Big| +1 \bigg| <\delta,$$
$$\bigg| \frac{d}{dt} \Big| f_t\big(\E'(i)\big) \cap B_r(z) \Big| \bigg| <\delta \qquad \forall i\neq h,k,$$
$$\bigg| \frac{d^2}{dt^2} \Big| f_t\big(\E'(i)\big) \cap B_r(z) \Big| \bigg| <C_0  \qquad \forall i=0,...,N\,.$$
(notice that $f_t(E) \cap B_r(z) = f_t(E \cap B_r(z))$ for every $E \subset \R^n$).
\item[(ii)]
if $E$ is a set of finite $s$-perimeter and $|t| <\e_1$, then
$$|P_s(f_t(E))- P_s(E)| \leq C_0|t| \,P_s(E).$$
\end{enumerate}
\end{lemma}
\begin{proof} Given $z\in \partial \E(h) \cap \partial \E(k)$ and $r>0$, let $T \in C^\infty_c(B_r(z))$ be the vector field given by Lemma~\ref{lemma:campo}, which satisfies
\begin{equation}
\label{eqn:lemma-div-appl}
 \int_{\E(h)} \Div (T) \, dx = 1 = - \int_{\E(k)} \Div (T) \, dx, \qquad
 \int_{\E(i)} \Div (T) \, dx =0 \quad \forall i \neq h,k.
\end{equation}
For every $t\in (0,1)$ we define $f_t(x) = x+ tT(x)$, $x\in \R^n$. Since $f_0(x) = x$ and ${\rm spt}\,  T \subset B_r(z)$, there exists $\e_1>0$ such that $\{ f_t\}_{|t|\leq \e_1}$ is a family of diffeomorphisms satisfying \eqref{eqn:supp-f-t}. By the area formula, for every Borel set $E\subset\R^n$% of finite $s$-perimeter,
$$\big| f_t(E) \cap B_r(z)\big| = \int_{E \cap B_r(z)} Jf_t(x) \, dx.$$
Noticing that $ Jf_t(x) =1+ t {\rm div} T(x) + O(t)$, we deduce that
$$\frac{d}{dt}\Big | _{t=0} \big| f_t(E) \cap B_r(z)\big| = \int_{E \cap B_r(z)} {\rm div}\, T(x)  \, dx$$
and statement (i) follows, possibly further reducing the value of $\e_1$, by \eqref{eqn:lemma-div-appl} and by the fact that $t \to \big| f_t(E) \cap B_r(z)\big|$ is a smooth function when $t$ is small.
By the change of variable formula we have also that
$$P_s( f_t(E))= \int_{E}\int_{E^c} \frac{Jf_t(x)Jf_t(y)}{|f_t(x)-f_t(y)|^{n+s}} \, dx\, dy.$$
Since $Jf_t(x) J_t(y) = 1+ t ( {\rm div}T(x)+{\rm div}T(y)) + o(t)$  there exist $C>0$ depending on $n$ and $T$ only such that $$| Jf_t(x) J_t(y)-1 | \leq C\,|t|;$$
moreover, up to considering larger values of $C$, we have
$$  \frac{1}{|f_t(x)-f_t(y)|^{n+s}} \leq  \frac{1}{(|x-y| - |t||T(x)-T(y)|)^{n+s}} \leq  \frac{1}{|x-y|^{n+s}(1 - |t|\|\nabla T\|_{L^\infty})^{n+s}}
 \leq  \frac{1+C|t|}{|x-y|^{n+s}}
$$
$$  \frac{1}{|f_t(x)-f_t(y)|^{n+s}} \geq  \frac{1}{(|x-y| + |t||T(x)-T(y)|)^{n+s}} \geq  \frac{1}{|x-y|^{n+s}(1 + |t|\|\nabla T\|_{L^\infty})^{n+s}}
 \geq  \frac{1-C|t|}{|x-y|^{n+s}} ,
$$
so that
$$\Big|  \frac{1}{|f_t(x)-f_t(y)|^{n+s}} - \frac{1}{|x-y|^{n+s}} \Big| \leq  \frac{C\,|t|}{|x-y|^{n+s}}
$$
for $t$ small enough. Hence, up to reducing $\e_1$ we deduce that
\begin{equation*}
\begin{split}
|&P_s( f_t(E))- P_s(E)| \leq  \int_{E}\int_{E^c} \Big| \frac{Jf_t(x)Jf_t(y)}{|f_t(x)-f_t(y)|^{n+s}} - \frac{1}{|x-y|^{n+s}} \Big| \, dx\, dy\\
&\leq \int_{E}\int_{E^c} \Big| \frac{Jf_t(x)Jf_t(y)}{|f_t(x)-f_t(y)|^{n+s}} - \frac{Jf_t(x)Jf_t(y)}{|x-y|^{n+s}} \Big| \, dx\, dy+ \int_{E}\int_{E^c} \Big| \frac{Jf_t(x)Jf_t(y)}{|x-y|^{n+s}} - \frac{1}{|x-y|^{n+s}} \Big| \, dx\, dy\\
& \leq C\,|t| \int_{E}\int_{E^c} \frac{1}{|x-y|^{n+s}} \, dx\, dy,
\end{split}
\end{equation*}
%$$|P_s( f_t(E))- P_s(E)| =\Big| \int_{E}\int_{E^c} \Big( \frac{Jf_t(x)Jf_t(y)}{|x-y|^{n+s}} - \frac{1}{|x-y|^{n+s}} \Big) \, dx\, dy\Big|
%.$$
which proves statement (ii).
\end{proof}

Lemma \ref{lemma:infinit-vol-fix} gives us a way to exchange volume between the chambers $\E(h)$ and $\E(k)$ at a point $z\in\pa\E(h)\cap\pa\E(k)$, without significantly change the volume of other chambers. The next step is choosing where to pick the points $z$ so to have enough freedom to achieve any small volume adjustment. To this end we introduce the following terminology: $\E(h)$ and $\E(k)$ are {\it neighboring chambers} if $\H^{n-1}(\partial \E(h) \cap \partial \E(k)) >0$.
%; they are {\it linked chambers} if there exists a sequence of neighboring chambers starting from $\E(h)$ and ending at $\E(k)$.
Let $S$ be the set of the indexes corresponding to neighboring chambers of $\E$,
$$
S= \Big\{(h,k) \in \{ 0,..., N\}^2: h<k, \; \H^{n-1}(\partial \E(h) \cap \partial \E(k)) >0 \Big\}\,,
$$
let $M \in \{N,..., 2N^2\}$ be the cardinality of $S$, and let $\phi = (\phi^1, \phi^2): \{1,..., M\} \to S$ be a bijection (so that $\phi$ is an enumeration of $S$). A finite family of distinct points $\{z_\alpha\}_{\a = 1,...,M}$ is a {\it system of interface points of $\E$} if for every $\alpha \in \{1,..., M\}$ we have that $z_\alpha \in \partial \E( \phi^1(\alpha) ) \cap \partial \E( \phi^2 (\alpha))$. The following lemma states the existence of a system of interface points of $\E$ and shows that a certain matrix, which keeps into account the links between different chambers, has rank $N$.

\begin{lemma}\label{lemma:linkedchambers}
 (i)  If $\E$ is an $N$-cluster in $\R^n$ and $M$ and $\phi$ are as above, then the
 matrix $L =(L_{j\a})_{j=0,...N,\, \a=1,...,M} \in \R^{(N+1) \times M}$ defined as
 $$
 L_{j\a} =
 \begin{cases}
 1 &\mbox{if } j= \phi^1(\a),
 \\
 -1 &\mbox{if } j= \phi^2(\a),
 \\
 0 &\mbox{if } j\neq \phi^1(\a), \phi^2(\a),
 \end{cases}
 \qquad 1\leq \a\leq M
 $$
 has rank $N$.

 \medskip

 \noindent (ii) If $\de>0$ and $A$ is an open set in $\R^n$ such that for every $h=0,...,N$ there exists a connected component $A'$ of $A$ with $|\E(0)\cap A'|>0$ and $|\E(h)\cap A'|>0$, then there exists systems of interface points $\{z_\alpha\}_{\a = 1,...,M}\subset A$ and $\{y_\alpha\}_{\a = 1,...,M}\subset A$ with $|z_\a-y_\b|>\de$ for every $\a,\b=1,...,M$.
\end{lemma}

\begin{proof}
 See \cite[Proof of Theorem 29.14, Step 1]{maggiBOOK}.
\end{proof}

By combining the previous lemma we obtain the following proposition on volume-fixing variations.

\begin{proposition}[Volume-fixing variations]\label{prop:vol-fix}
 Let $s\in(0,1)$, $\E$ be an $N$-cluster with $0<|\E(h)|<\infty$ for every $h=1,...,N$,  $\{z_\alpha\}_{\a = 1,...,M}$ be a system of interface points of $\E$, and let $0<r<\min\{|z_\a-z_\b|/4:1\le\a<\b\le M\}$.% according to Lemma~\ref{lemma:linkedchambers}.

 Then there exist positive constants $\eta, \e_1,\e_2, C$ ($s$, $\E$, $\{z_\a\}_{\a =1,...,M}$ and $r$) with the following property:
 for every $N$-cluster $\E'$ with $d(\E, \E') < \e_2$ there exists a $C^1$-function
 $$\Phi : ((-\eta, \eta)^{N+1} \cap V) \times \R^n \to \R^n$$
  such that
 \begin{enumerate}
\item[(i)] if $\bold a \in (-\eta,\eta)^{N+1}\cap V$ then $\Phi(\bold a, \cdot): \R^n \to \R^n$ is a diffeomorphism with %there exists $M\in \N$ with $N \leq M \leq 2N^2$ and a finite family $\{z_\alpha\}_{\a = 1,...,M}$ of interface points of $\E$ with $|z_\a-z_\b|>4\e_1$ for $1\leq \a<\b \leq M$, such that
$$\{ x\in \R^n : \Phi(\bold a,x) \neq x\} \subset  \bigcup_{\a=1}^M B_r(z_\a) \cc \R^n$$
\item[(ii)] if $\bold a \in (-\eta,\eta)^{N+1}\cap V$ then for $0\leq h\leq N$
$$\Big| \Phi(\bold a, \E'(h)) \cap \{ x\in \R^n : \Phi(\bold a,x) \neq x\}\Big| = \Big| \E'(h) \cap \{ x\in \R^n : \Phi(\bold a,x) \neq x\}\Big| + \bold a (h);$$
\item[(iii)] if $\bold a \in (-\eta, \eta)^{N+1} \cap V$ and $F$ is a set of finite $s$-perimeter, then
$$|P_s(\Phi(\bold a , F))- P_s(F)| \leq CP_s(F) \sum_{h=0}^N |\bold a (h)|.$$
\end{enumerate}
\end{proposition}

\begin{proof} Given Lemma \ref{lemma:infinit-vol-fix} and Lemma \ref{lemma:linkedchambers} the proof is basically the same as in \cite[Proof of Theorem 29.14]{maggiBOOK}, so we just give a sketch for the sake of clarity.
By Lemma~\ref{lemma:infinit-vol-fix} given positive constants $\de$ and $r$, there exist positive constants $\e_1, \e_2, C_0$ (depending on $\E$, $r$, $\de$ and $\{z_\a\}_{\a=1}^M$) and diffeomorphisms $\{f^\a_t\}_{\a=1,...,M\,,|t|<\e_1}$ such that
\begin{equation}
\label{eqn:supp-f-t-appl}
\big\{ x\in \R^n: x \neq f^\a_t(x) \big\} \subset\subset B_r(z_\a), \qquad \forall |t| \le \e_1,\a=1,...,M\,,
\end{equation}
and, if $\E'$ is a cluster with $d(\E, \E') < \e_2$, $|t| <\e_1$, $\a=1,...,M$, and $(h,k)=(\phi^1(\a),\phi^2(\a))$, then
\begin{equation}
\label{eqn:vol-fix-appl1}\Big| \frac{d}{dt} \Big| f^\a_t\big(\E'(h)\big) \cap B_r(z_\a) \Big| -1 \Big| <\delta, \qquad \Big| \frac{d}{dt} \Big| f^\a_t\big(\E'(k)\big) \cap B_r(z_\a) \Big| +1 \Big| <\delta,
\end{equation}
\begin{equation}
\label{eqn:vol-fix-appl2}\Big| \frac{d}{dt} \Big| f^\a_t\big(\E'(i)\big) \cap B_r(z_\a) \Big| \Big| <\delta \qquad \mbox{for } i\neq h,k,
\end{equation}
\begin{equation}
\label{eqn:vol-fix-appl3}
\Big| \frac{d^2}{dt^2} \Big| f^\a_t\big(\E'(i)\big) \cap B_r(z_\a) \Big| \Big| <C_0  \qquad \mbox{for } 0 \leq i\leq N
\end{equation}
and such that, whenever $E$ is a set of finite $s$-perimeter,
\begin{equation}
\label{eqn:vol-fix-appl4}|P_s(f^\a_t(E))- P_s(E)| \leq C_0|t| \,P_s(E)\,.
\end{equation}
Since $r<\min\{|z_\a-z_\b|/4:1\le\a<\b\le M\}$, if we define $\Psi: (-\e_1,\e_1)^M \times \R^n \to \R^n$ by setting
$$
\Psi(\bold t, x) = (f^1_{t_1} \circ f^2_{t_2} \circ ... \circ f^M_{t_M})(x) , \qquad (\bold t,x) \in (-\e_1,\e_1)^M \times \R^n\,,
$$
then $\Psi(\bold t,\cdot)$ is a diffeomorphisms with $\{\Psi(\bold t,\cdot)\ne\Id\}$ compactly contained in the union of the disjoint balls $\{B_r(z_\a)\}_{\a=1}^M$. We claim the existence of $\eta>0$ and $\zeta:(-\eta,\eta)^{N+1} \cap V \to \R^M$ such that
\begin{equation}
\label{eqn:def-Phi}
\Phi(\bold a, x)=\Psi(\zeta(\bold a), x) \qquad (\bold a,x) \in ((\eta,\eta)^{N+1} \cap V) \times \R^n,
\end{equation}
satisfies all the required properties. To this end, we consider first the function $\psi: (-\e_1,\e_1)^M \to V \subseteq \R^{N+1}$ defined by setting, for every $h=0,...,N$ and $\bold t\in (-\e_1,\e_1)^M$,
\begin{equation}
\begin{split}
\psi_h(\bold t) &= \big| \Psi(\bold t, \E'(h)) \cap \big\{ x\in \R^n: x \neq \Psi(\bold t, x) \big\} \big| - \big| \E'(h) \cap \big\{ x\in \R^n: x \neq \Psi(\bold t, x) \big\} \big|\\
&= \sum_{\a=1}^M \big| f^\a_{t_\a}\big(\E'(h)\big) \cap B_r(z_\a) \big| -\big| \E'(h) \cap B_r(z_\a) \big|.
\end{split}
\end{equation}
By \eqref{eqn:vol-fix-appl1}, \eqref{eqn:vol-fix-appl2}, \eqref{eqn:vol-fix-appl3}, we see that $\psi(0) = 0$, $|\nabla ^2 \psi(\bold t ) | \leq C_0$ for every $\bold t \in (-\e_1, \e_1)^M$, with $|\partial_\a \psi_h(0) - L_{h\a}|\le C(N,M)\,\de$ for every $h=0,...,N$ and $\a=1,...,M$. Since the rank of $(L_{h\a})_{h,\a}$ is $N$ (Lemma \ref{lemma:linkedchambers}), by arguing as in \cite[Proof of Theorem 29.14, Step 3]{maggiBOOK} we find that provided $\de$ is small enough then there exists $\kappa>0$ such that $\nabla \psi(0) e \geq \kappa |e|$ for every $e \in \ker \nabla \psi(0)^\perp$.
By the implicit function theorem (with the same statement as in \cite[Proof of Theorem 29.14, Step 2]{maggiBOOK} for having a quantitative dependence of $\eta$ on $\E$ and $\e_2$ but not on $\E'$) we deduce that there exists a class $C^2$ function $\zeta: (-\eta,\eta)^{N+1} \cap V \to \R^M$ such that
$$\psi(\zeta(\bold a )) = \bold a, \qquad |\zeta(\bold a)| \leq \frac{2}{\kappa} |\bold a |.$$
With this definition at hand, it is clear that $\Phi$ defined in \eqref{eqn:def-Phi} satisfies (i). Thanks to the definition of $\zeta$ and $\psi$, it satisfies also (ii). We are left to check (iii), which requires a computation specific to the fractional setting. If $\bold a \in (-\eta, \eta)^{N+1} \cap V$ and $F$ is a set of finite $s$-perimeter, then we have
\begin{equation}
\label{eqn:per-cons}
\begin{split}
&|P_s(\Phi(\bold a , F))- P_s(F)| =
|P_s((f^1_{\zeta_1(\bold a)} \circ ... \circ f^M_{\zeta_M(\bold a)})(F) )- P_s(F)|\\
&=\sum_{\a=1}^{M-1}|P_s((f^\a_{\zeta_\a(\bold a)} \circ ... \circ f^M_{\zeta_M(\bold a)})(F))- P_s((f^{\a+1}_{\zeta_{\a+1}(\bold a)} \circ ... \circ f^M_{\zeta_M(\bold a)})(F))| + |P_s(f^M_{\zeta_M(\bold a)}(F))- P_s(F)|.
%\leq
%\leq C_1P_s(F) \sum_{h=0}^N |\bold a (h)|,
\end{split}
\end{equation}
By \eqref{eqn:vol-fix-appl4}, we deduce that for every $\a=1,...,M-1$
\begin{equation}
\label{eqn:cons-per}
|P_s((f^\a_{\zeta_\a(\bold a)} \circ ... \circ f^M_{\zeta_M(\bold a)})(F))- P_s((f^{\a+1}_{\zeta_{\a+1}(\bold a)} \circ ... \circ f^M_{\zeta_M(\bold a)})(F))| \leq C_0 |\zeta_\a(\bold a)| P_s((f^{\a+1}_{\zeta_{\a+1}(\bold a)} \circ ... \circ f^M_{\zeta_M(\bold a)})(F))
\end{equation}
and similarly
$$ |P_s(f^M_{\zeta_M(\bold a)}(F))- P_s(F)| \leq  C_0|\zeta_M(\bold a)| P_s(F).$$
In particular, for every $\a=1,...,M-1$, since $|\zeta_\a(\bold a)| \leq \e_1 \leq 1$, we obtain
\begin{equation*}
\begin{split}
&P_s((f^\a_{\zeta_\a(\bold a)} \circ ... \circ f^M_{\zeta_M(\bold a)})(F)) \\&\leq P_s((f^{\a+1}_{\zeta_{\a+1}(\bold a)} \circ ... \circ f^M_{\zeta_M(\bold a)})(F))+ |P_s((f^\a_{\zeta_\a(\bold a)} \circ ... \circ f^M_{\zeta_M(\bold a)})(F))- P_s((f^{\a+1}_{\zeta_{\a+1}(\bold a)} \circ ... \circ f^M_{\zeta_M(\bold a)})(F))|
\\
&\leq (1+C_0) P_s((f^{\a+1}_{\zeta_{\a+1}(\bold a)} \circ ... \circ f^M_{\zeta_M(\bold a)})(F))
\end{split}
\end{equation*}
and
$$P_s(f^M_{\zeta_M(\bold a)}(F))\leq  (1+C_0)P_s(F);$$
an easy induction shows then that
\begin{equation}
\label{eqn:cons-per2}
P_s((f^\a_{\zeta_\a(\bold a)} \circ ... \circ f^M_{\zeta_M(\bold a)})(F))
\leq (1+C_0)^M P_s(F).
\end{equation}
By \eqref{eqn:per-cons}, \eqref{eqn:cons-per}, and \eqref{eqn:cons-per2}, we deduce that
\begin{equation}
\label{eqn:fine-vol-fix}
\begin{split}
|P_s(\Phi(\bold a , F))- P_s(F)|
&=C_0 \sum_{\a=1}^{M-1} |\zeta_\a(\bold a)| P_s((f^{\a+1}_{\zeta_{\a+1}(\bold a)} \circ ... \circ f^M_{\zeta_M(\bold a)})(F)) +C_0 |\zeta_M(\bold a)| P_s(F)
\\
&\leq (1+C_0)^{M+1}P_s(F) \sum_{\a=1}^{M} |\zeta_\a(\bold a)|\leq \frac{2M^{1/2}(1+C_0)^{M+1}}{\kappa}P_s(F) \sum_{h=0}^N |\bold a (h)|,
\end{split}
\end{equation}
so that also (iii) is satisfied.
\end{proof}

In the local case Proposition \ref{prop:vol-fix} would be sufficient for showing that isoperimetric clusters are locally almost-minimizing perimeter (a key step in the regularity theory) and for modifying minimizing sequences in the existence argument. In the fractional case, the latter application will need the following version of Proposition \ref{prop:vol-fix}.

\begin{proposition}[Volume-fixing variations of a minimizing sequence]\label{prop:vol-fix-k}
 Let $s\in(0,1)$, $m\in\R^N_+$, $\{\E_k\}_{k\in \N}$ be a sequence of $N$-clusters with $m(\E_k) = m$ for every $k\in\N$, and define $S>0$ by setting
 \[
 \omega_n S^n = 2(m(1)+...+m(N))\,.
 \]
 Finally, let us assume that there exist $c_0>0$ and sequences $\{x_k(1)\}_{k\in \N},...,\{x_k(N)\}_{k\in \N}$ such that
 \begin{equation}
 \label{hp:some-mass}
 | \E_k(h) \cap B_S(x_k(h))| \geq c_0 \qquad \mbox{for every } k\in \N \mbox{ and } h=1,...,N.
 \end{equation}
 Then there exist positive constants $\eta, C$ such that for every $k\in \N$ (up to a not relabeled subsequence) there exists a $C^1$-function
 $$\Phi_k : ((-\eta, \eta)^N \cap V) \times \R^n \to \R^n$$
  such that
 \begin{enumerate}
\item[(i)] if $\bold a \in (-\eta,\eta)^{N+1}\cap V$ then $\Phi_k(\bold a, \cdot): \R^n \to \R^n$ is a diffeomorphism with %there exists $M\in \N$ with $N \leq M \leq 2N^2$ and a finite family $\{z_\alpha\}_{\a = 1,...,M}$ of interface points of $\E$ with $|z_\a-z_\b|>4\e_1$ for $1\leq \a<\b \leq M$, such that
$$\{ x\in \R^n : \Phi_k(\bold a,x) \neq x\} \subset  \bigcup_{h=0}^N B_S(x_k(h)) \cc \R^n$$
\item[(ii)] if $\bold a \in (-\eta,\eta)^{N+1}\cap V$ then for $0\leq h\leq N$
$$\Big| \Phi_k(\bold a, \E_k(h)) \cap \{ x\in \R^n : \Phi_k(\bold a,x) \neq x\}\Big| = \Big| \E_k(h) \cap \{ x\in \R^n : \Phi_k(\bold a,x) \neq x\}\Big| + \bold a (h);$$
\item[(iii)] if $\bold a \in (-\eta, \eta)^{N+1} \cap V$ and $F$ is a set of finite $s$-perimeter, then
$$|P_s(\Phi_k(\bold a , F))- P_s(F)| \leq CP_s(F) \sum_{h=0}^N |\bold a (h)|.$$
\end{enumerate}
\end{proposition}

%The proof mixes the ideas developed for the previous proposition and an idea of the classical proof of existence of minimizing clusters. More precisely, we will construct nontrivial limits of some portions of $\E_k$ (thanks to \eqref{hp:some-mass}) and use these limits as reference clusters for our constructions.
In the course of the proof we shall need the following basic property of fractional perimeters: for every measurable set $E$ and for every ball $B$ it holds
\begin{equation}
\label{eqn:per-e-palla}
P_s(E\cap B) \leq I_s(E \cap B, E^c) + I_s(E\cap B, B^c) \leq I_s(E, E^c) + I_s(B, B^c) =  P_s(E) + P_s(B).
\end{equation}

\begin{proof}
Up to extracting a not relabelled subsequence, we may assume that there exist $\lim_{k \to \infty} x_k(h) - x_k(h') $ for every $h, h' \in \{1,..., N\}$. Moreover, we can  partition $\{1,...,N\}$ into $\ell$ disjoint sets $\Lambda_1,..., \Lambda_\ell$ such that for every $j=1,...,\ell$
$$\mbox{there exists }\lim_{k \to \infty} x_k(h) - x_k(h') \in \overline{B_{2NS} }\qquad \mbox{if } h, h' \in \Lambda_j,$$
$$\liminf_{k \to \infty} x_k(h) - x_k(h') > 2S%\in  \R^n \setminus B_{2S}
\qquad \mbox{for every } h, h' \in \Lambda_j.$$
The construction of the sets $\Lambda_j$ is performed in \cite[Section 29.7, Step 1]{maggiBOOK}. Then we have isolated $\ell$ disjoint nuclei in $\E_k$, each of them of the form
$$\E_k(h) \cap \bigcup_{h' \in \Lambda_j} B_S(x_k(h')) \qquad \mbox{for every } h = 1,..., N, \; j=1,..., \ell.$$
By setting $v_j = 8 (N+1)S j e_n$ and by selecting an element $h_j$ in each set $\Lambda_j$, we define a new sequence of clusters $\E_k^*$ by setting for every $h=1,...,N$
$$\E_k^*(h) = \bigcup_{i=1}^\ell \left( v_j - x_k(h_j) + \Big( \E_k(h) \cap \bigcup_{h' \in \Lambda_j} B_S(x_k(h'))
\Big) \right).
$$
For every $h=1,...,N$, by \eqref{eqn:per-e-palla} we obtain
\begin{equation*}
\begin{split}
P_s(\E^*_k(h)) &\leq \sum_{j=1}^\ell \sum_{h'\in \Lambda_j} P_s\big(\E^*_k(h) \cap B_S(v_j- x_k(h_j)+x_k(h')) \big) \\
&= \sum_{j=1}^\ell \sum_{h'\in \Lambda_j} P_s\big(\E_k(h) \cap B_S(x_k(h')) \big) = \sum_{h'=1}^N P_s\big(\E_k(h) \cap B_S(x_k(h')) \big)
\\
&\leq NP_s \big ( B_S(0) \big)+ \sum_{h=1}^N P_s \big(\E_k(h)\big).
\end{split}
\end{equation*}
By the bound on the perimeters of $\E_k^*$ above, which are all contained in $B_{12(N+1) S}(0)$, we deduce that there exists a cluster $\E \subset B_{12(N+1) S}(0)$ such that, up to a subsequence, each chamber of $\E_k^*$ converges to the corresponding chamber of $\E^*$ in $L^1(B_{12(N+1) S}(0))$. Moreover, by \eqref{hp:some-mass}, if $h \in \Lambda_j$ for some $j$, we have that
$$|\E^*(h) \cap B_S(v_j- x_k(h_j)+x_k(h))| \geq c_0.$$
We apply Lemma~\ref{lemma:linkedchambers} to obtain a system of interface points for $\E^*$ in  $\cup_{h=1}^N B_S(v_j- x_k(h_j)+x_k(h))$ (we use the open set $A$ given by a union of balls). Following the proof of Proposition~\ref{prop:vol-fix} applied to the reference cluster $\E$, we find $\eta,\e_1$ and $C_1$ (independent on $k$), one-parameter families of diffeomorphisms $\{f^\a_t\}_{\a=1,...,M}$ and $\zeta:(-\eta,\eta)^{N+1} \cap V \to \R^M$ (the latter two depend on $k$, as in the previous proof they depended on $\E'$, but for simplicity we omit this dependence) with the following properties. For every $\a =1,...,M$ there exists a $j\in \{1,...,\ell\}$ and $h' \in \Lambda_j$ such that
\begin{equation}
\label{eqn:supp-f-t-appl-new}
\big\{ x\in \R^n: x \neq f^\a_t(x) \big\} \subset\subset B_S(v_j- x_k(h_j)+x_k(h')), \qquad \mbox{for every } |t| \le \e_1,
\end{equation}
the sets $\big\{ x\in \R^n: x \neq f^\a_t(x) \big\}$ are all disjoint as $\alpha$ ranges in $1,... , M$,
\begin{equation}
\label{eqn:vol-fix-appl4-new}|P_s(f^\a_t(E))- P_s(E)| \leq C_0|t| \,P_s(E),
\end{equation}
and setting
\begin{equation}
\label{eqn:def-Phi-new}
\Phi^*_k(\bold a, x)=(f^1_{\zeta_1(\bold a)} \circ f^2_{\zeta_2(\bold a)} \circ ... \circ f^M_{\zeta_M(\bold a)})(x) \qquad (\bold a,x) \in ((\eta,\eta)^{N+1} \cap V) \times \R^n,
\end{equation}
we have
\begin{equation}
\label{eqn:vol-change-star}\Big| \Phi^*_k(\bold a, \E^*_k(h)) \cap \{ x\in \R^n : \Phi^*_k(\bold a,x) \neq x\}\Big| = \Big| \E^*_k(h) \cap \{ x\in \R^n : \Phi^*_k(\bold a,x) \neq x\}\Big| + \bold a (h).
\end{equation}
Now we suitably translate the functions $f^1_{\zeta_1(\bold a)} , ... ,f^M_{\zeta_M(\bold a)}$ in such a way that they act on the cluster $\E_k$ rather than on its translation $\E^*_k$; more precisely, we define for every $\a =1,...,M$
$$
g^\a_{\zeta_\a(\bold a)}(x) = f^\a_{\zeta_\a(\bold a)}(x + v_j- x_k(h_j)) - v_j+ x_k(h_j) \qquad
$$
(once more we omit the dependence on $k$ for ease of notation; here $j\in \{1,...,\ell\}$ and $h' \in \Lambda_j$ are chosen to satisfy \eqref{eqn:supp-f-t-appl-new}) and for every $k\in \N$
$$\Phi_k(\bold a, x)=(g^1_{\zeta_1(\bold a)} \circ g^2_{\zeta_2(\bold a)} \circ ... \circ g^M_{\zeta_M(\bold a)})(x) \qquad (\bold a,x) \in (\eta,\eta)^{N+1} \times \R^n.
$$
It is clear that, since $f^\a_{\zeta_\a(\bold a)}$ is the identity outside $B_S(v_j- x_k(h_j)+x_k(h))$, the diffeomorphism $g^\a_{\zeta_\a(\bold a)}$ is the identity outside $B_S(x_k(h))$; moreover
$$\{ x\in \R^n: x \neq g^\a_t(x) \big\} =x_k(h_j)-v_j + \{ x\in \R^n: x \neq f^\a_t(x) \big\}.$$
It is easily checked by the definition of $\E^*_k$ that for every $h=1,...,N$ the set
$g^\a_{\zeta_\a(\bold a)}( \E_k(h)) \cap \big\{ x\in \R^n: x \neq g^\a_t(x) \big\}$ is a translation of $f^\a_{\zeta_\a(\bold a)}( \E_k^*(h)) \cap \big\{ x\in \R^n: x \neq f^\a_t(x) \big\}$, so that the volume change induced on $\E_k^*$ by $f^\a_{\zeta_\a(\bold a)}$ is the same volume change induced on $\E_k$ by $g^\a_{\zeta_\a(\bold a)}$: in other words,
$$\big| g^\a_{\zeta_\a(\bold a)}( \E_k(h)) \cap \big\{ x\in \R^n: x \neq g^\a_t(x) \big\} \big| = \big| f^\a_{\zeta_\a(\bold a)}( \E_k^*(h)) \cap \big\{ x\in \R^n: x \neq f^\a_t(x) \big\} \big|.$$
Since the diffeomorphisms $f^\a_{\zeta_\a(\bold a)}$ act (as $\a$ varies) on nonintersecting sets, and the same happens to $g^\a_{\zeta_\a(\bold a)}$, by composing the diffeomorphisms when $\a$ varies by \eqref{eqn:vol-change-star} we deduce that
\begin{equation*}
\begin{split}
\Big| \Phi_k(\bold a, \E_k(h)) \cap \{ x\in \R^n : \Phi_k(\bold a,x) \neq x\}\Big|
&=\Big| \Phi^*_k(\bold a, \E^*_k(h)) \cap \{ x\in \R^n : \Phi^*_k(\bold a,x) \neq x\}\Big|
\\ & = \Big| \E^*_k(h) \cap \{ x\in \R^n : \Phi^*_k(\bold a,x) \neq x\}\Big| + \bold a (h)
\\ & = \Big| \E_k(h) \cap \{ x\in \R^n : \Phi_k(\bold a,x) \neq x\}\Big| + \bold a (h);
\end{split}
\end{equation*}
hence (ii) holds true. To prove (iii), we repeat word by word the argument between \eqref{eqn:per-cons} and \eqref{eqn:fine-vol-fix} with $g^\a_{\zeta_\a(\bold a)}$ replacing $f^\a_{\zeta_\a(\bold a)}$ at every occurrence (by the nonlocality of the $s$-perimeter, the fact that (iii) holds with $\Phi^*_k$ replacing $\Phi_k$ does not allow directly to conclude the statement; we need to repeat the argument on each $g^\a_{\zeta_\a(\bold a)}$).
\end{proof}

\subsection{Truncation lemma}\label{section truncation} We now state and prove the truncation lemma for fractional perimeters needed in the existence proof. In the case of sets ($N=1$) this lemma has already appeared as \cite[Lemma 4.5]{F2M3}.

 \begin{lemma} \label{lemma truncation}
 Let $n\ge 2$, $s\in(0,1)$, $\tau\in(0,1)$, let $\E$ be an $N$-cluster in $\R^n$, and $F \subseteq \R^n$ be a closed set with $u(x) = \dist(x,F)$ for $x\in\R^n$. If
 \[
 \sum_{h=1}^{N} |\E(h)\setminus F|\le\tau\,,
 \]
 then there exists $ r_0 \in [0, C_1\,\tau^{1/n}]$ such that the $N$-cluster $\E'$ in $\R^n$ defined by
 $$\E'(h) = \E(h) \cap \{ u\leq r_0\} \qquad 1 \leq h \leq N$$
satisfies
 \begin{equation}
   \label{tr1}
 (1-s)\,P_s(\E' )\le (1-s)\,P_s(\E)-\frac{\dist(\E, \E')}{C_2(n,s)\,\tau^{s/n}}\,,
 \end{equation}
 where
 \begin{equation}
   \label{c1c2}
    C_1(n,s):=2^{1+(n-s)/s}\,\Big(\frac{4\,|B|^{(n-s)/n}\,P(B)}{s\,(1-s)\,P_s(B)}\Big)^{1/s} \,,\qquad C_2(n,s):=\frac{2|B|^{(n-s)/n}}{(1-s)\,P_s(B)}\,.
 \end{equation}
 In particular, $\sup\{C_1(n,s)+C_2(n,s):s_0\le s<1\}<\infty$ for every fixed $s_0\in(0,1)$.
 \end{lemma}

\begin{remark}
  {\rm Here we pay some attention to the dependency of constants from $s$, as the constants can be shown to be uniform in the limit $s\to 1^-$.}
\end{remark}

 \begin{proof}
For every $r\geq 0$, let us call $F_r =  \{ u\leq r\}$ the $r$-enlargement of $F$ and let us define the cluster $\E^r$ whose chambers are $\E^r(h) =  \E(h) \cap F_r$ for every $1 \leq h \leq N$ . Without loss of generality we may assume that
\[
\sum_{h=1}^{N}|\E(h)\setminus F_{C_1\,\tau^{1/n}}|>0
\]
otherwise, we set $r_0 =C_1\,\tau^{1/n}$ and \eqref{tr1} holds. If we set $m(r)=\sum_{h=1}^{N}|\E(h)\setminus F_r|$, $r>0$, then $m$ is a nonincreasing function with
 \begin{equation}
     \label{golden2}
        [0,C_1\,\tau^{1/n}]\subset{\rm spt}\,m\,\qquad m(0)\le \tau\,,\qquad m'(r)=-\sum_{h=1}^{N}\H^{n-1}(\E(h)\cap\pa F_r)\quad\mbox{for a.e. $r>0$}\,.
 \end{equation}
 Arguing by contradiction, we now assume that
 \begin{equation}
 \label{tr3}
 (1-s)\,P_s(\E)\le (1-s)\,P_s(\E^r)+\frac{m(r)}{C_2\,\tau^{s/n}}\,,\qquad\forall r\in(0,C_1\,\tau^{1/n})\,.
 \end{equation}
 First, for every $r>0$ and $h=1,...,N$ we have the identity
 \begin{eqnarray*}
  P_s(\E(h)\cap F_r)-P_s(\E(h))
  &=& 2  P_s(F_r;\E(h)) - P_s(\E(h)\setminus F_r) \\
  &=& 2\,\int_{\E(h)\cap F_r}\int_{\E(h)\cap F_r^c}\frac{dx\,dy}{|x-y|^{n+s}}-P_s(\E(h)\setminus F_r)\,.
 \end{eqnarray*}
Since $\E(h)\cap F_r \subseteq B_{u(y)-r}(y)$ for every $y \in \E(h)\cap F_r^c$ and by the coarea formula, for every $r>0$ we estimate
\begin{eqnarray*}
 \int_{\E(h)\cap F_r}\int_{\E(h)\cap F_r^c}\frac{dx\,dy}{|x-y|^{n+s}}
 &\le&
 \int_{\E(h)\cap F_r^c}dy \int_{ B_{u(y)-r}(y)}\frac{dx}{|x-y|^{n+s}}
 \\
 &=& P(B) \int_{\E(h)\cap F_r^c}dy \int_{u(y)-r}^\infty\frac{d\rho}{\rho^{1+s}}
\\
&=& \frac{P(B)}{s} \int_{\E(h)\cap F_r^c} \frac{dy}{(u(y)-r)^s}
=\frac{P(B)}s\,\int_r^\infty\,\frac{-m'(t)}{(t-r)^s}\,dt\,.
\end{eqnarray*}
Finally, by the nonlocal isoperimetric inequality,
$$\sum_{h=1}^N P_s(\E(h)\setminus F_r)\ge P_s\Big(\bigcup_{h=1}^N\E(h)\setminus F_r\Big) \ge P_s(B)|B|^{(s-n)/n}\,m(r)^{(n-s)/n}.$$
 We may thus combine these three remarks with \eqref{tr3} to conclude that, if $r\in(0,C_1\,\tau^{1/n})$, then
 \begin{eqnarray}\nonumber
     0&\le& \frac{2\,P(B)}s\,\int_r^\infty\frac{-m'(t)}{(t-r)^s}\,dt-\frac{P_s(B)}{|B|^{(n-s)/n}}\,m(r)^{(n-s)/n}+\frac{m(r)}{(1-s)\,C_2\,\tau^{s/n}}
     \\\label{golden4}
     &\le&\frac{2\,P(B)}s\,\int_r^\infty\frac{-m'(t)}{(t-r)^s}\,dt-\frac{P_s(B)}{2|B|^{(n-s)/n}}\,m(r)^{(n-s)/n}\,,
 \end{eqnarray}
 where in the last inequality we have used our choice of $C_2$ and the fact that $m(r)\le\tau$ for every $r>1$. We rewrite \eqref{golden4} in the more convenient form
 \begin{equation}
   \label{golden8}
   m(r)^{(n-s)/n}\le C_3\,\int_r^\infty\frac{-m'(t)}{(t-r)^s}\,dt\,,\qquad \forall r\in(1,1+C_1\,\tau^{1/n})\,,
 \end{equation}
 where we have set
 \[
 C_3(n,s):=\frac{4\,|B|^{(n-s)/n}\,P(B)}{s\,P_s(B)}\,.
 \]
Proceeding as in \cite[Lemma 4.5]{F2M3} one can show that any function $m$ satisfying the previous differential inequality satisfies $m(r)\to 0$ as $r \to C_1\,\tau^{1/n}$. This gives a contradiction.
 \end{proof}

\subsection{Nucleation lemma}\label{section nucleation}
The following nucleation lemma is obtained by combining part of the argument leading to its local analogous (see \cite[Lemma 29.10]{maggiBOOK}) with a lemma for fractional perimeters already appeared in \cite[Lemma 4.3]{F2M3}.

 \begin{lemma}\label{lemma:nucleation}
  Let $n\ge 2$ and $s\in(0,1)$. %there exists a positive constant $c:=c(n,s)$ with the following property:
  If $P_s(E)<\infty$, $0<|E|<\infty$, and
  \begin{equation}
  \label{eqn:eps}
  \e \leq \min \Big\{ |E|, \frac{1-s}{\chi_1\chi_2} P_s(E)\Big\}%^{n/s}
  \end{equation}
  then there exists a finite family of points $I \subset \R^n$ such that
  $$
  \Big| E \setminus \bigcup_{x\in I} B_2(x)\Big| <\e,
  $$
  \begin{equation}
  \label{eqn:density-nucleata}
  \Big| E \cap B_1(x)\Big| \geq \Big(\frac{\chi_1\,\e}{(1-s)\,P_s(E)}\Big)^{n/s} \qquad \forall x\in I,
  \end{equation}
    where
 \begin{equation}
 \label{defn:xi12}
    \chi_1(n,s):=\frac{(1-s)\,P_s(B)}{4\,|B|^{(n-s)/n}\,\xi(n)}\,,
\qquad
    \chi_2(n,s):=\frac{2^{3+(n/s)}\,|B|^{(n-s)/n}\,P(B)}{s(1-s)\,P_s(B)}\,,
 \end{equation}
  and where $\xi(n)$ is Besicovitch's covering constant (see for instance \cite[Theorem 5.1]{maggiBOOK}). In particular, $0<\inf\{\chi_1(n,s),\chi_2(n,s)^{-1}:s\in[s_0,1)\}<\infty$ for every $s_0\in(0,1)$.
Moreover, $|x-y| >2$ for every $x,y \in I$, $x\neq y$, and
$$ \# I \leq |E| \Big(\frac{(1-s)\,P_s(E)}{\chi_1\,\e}\Big)^{n/s} .$$
 \end{lemma}
\begin{proof}
In \cite[Proof of Lemma 4.3, Step 1]{F2M3} it is proved that if $x\in E^{(1)}$ with
   \begin{equation}
     \label{rutgers1}
     |E\cap B_1(x)|\le\Big(\frac{(1-s)\,P_s(B)}{2\,|B|^{(n-s)/n}\,\a} \Big)^{n/s}
   \end{equation}
   for some $\a$ satisfying
   \begin{equation}
     \label{alpha}
     \a\geq\frac{2^{2+(n/s)}\,P(B)}s\,,
   \end{equation}
   then there exists $r_x\in(0,1]$ such that
   \begin{equation}\label{rutgers2}
   |E\cap B_{r_x}(x)|\le\frac{(1-s)}\a\,\int_{E\cap B_{r_x}(x)}\int_{E^c}\frac{dz\,dy}{|z-y|^{n+s}}\,.
   \end{equation}
This statement %, which is the exact nonlocal version of \cite[Proof of Lemma 29.10, Step 2]{Ma},
 is in turn the basic step for proving the following claim: if $F \subseteq \R^n$ is closed, $\e$ satisfies \eqref{eqn:eps}, and
\begin{equation}
\label{eqn:contr-F}
\Big| \big\{ x\in E : \dist(x,F)>1\big\} \Big| \geq \e,
\end{equation}
then there exists $x\in E^{(1)}$ with $\dist(x,F) >1$ and
$$\Big| E \cap B_1(x)\Big| \geq \Big(\frac{\chi_1\,\e}{(1-s)\,P_s(E)}\Big)^{n/s}.$$ Indeed, by contradiction, assume that if $x\in E^{(1)}$ with $\dist(x,F) >1$ then
$$\Big| E \cap B_1(x)\Big| < \Big(\frac{\chi_1\,\e}{(1-s)\,P_s(E)}\Big)^{n/s}= \Big(\frac{(1-s)\,P_s(B)}{2\,|B|^{(n-s)/n}\,\a} \Big)^{n/s}.$$
In the last equality we chose $\alpha = 2(1-s) P_s(E) \xi(n)/\e$. Thanks to our assumption \eqref{eqn:eps} on $\e$, we see that \eqref{alpha} holds.
Hence, by \eqref{rutgers2} for every $x\in E^{(1)}$ with $\dist(x,F) >1$ there exists $r_x$ such that \eqref{rutgers2} holds.
Applying the Besicovitch covering theorem to $\F = \{ \overline{B_{r_x}(x)}: x\in E^{(1)}, \; \dist(x,F) > 1\}$ we find a countable disjoint subfamily $\F'$ of $\F$ such that
\begin{equation*}
\begin{split}
\Big| \big\{ x\in E : \dist(x,F)>1\big\} \Big|
& \leq
\xi(n) \sum_{\overline{B_{r_x}(x)} \in \F'} |E\cap B_{r_x}(x)|
\\& \leq \frac{(1-s)\xi(n)}\a \sum_{\overline{B_{r_x}(x)} \in \F'}  \int_{E\cap B_{r_x}(x)}\int_{E^c}\frac{dz\,dy}{|z-y|^{n+s}}
\\&
\leq \frac{(1-s)\xi(n)}\a P_s(E).
\end{split}
\end{equation*}
Thanks to our choice of $\alpha$ and to \eqref{eqn:eps}, the right-hand side equals $\e/2$ and this contradicts \eqref{eqn:contr-F}.

Finally, we define $\{x_i\}_{i \in I}$ inductively. First, we define $x_1$ applying the claim with $F=\emptyset$. Then, inductively, we assume that we have chosen $I = \{ x_i\}_{i=1,...,s}$ and we consider whether
$$\Big| E \setminus \bigcup_{x\in I} B_2(x)\Big| <\e$$
holds or not. If this holds, the set $I$ satisfies the properties required by our lemma; otherwise, we apply the claim with $F = \cup_{j=1}^i \overline{B_1(x_j)}$, to find $x_{s+1}$ such that \eqref{eqn:density-nucleata} holds and such that its distance from $\{ x_1,..., x_s\}$ is at least $2$. Since $|E|<\infty$, this process ends in finitely many steps.
\end{proof}

\subsection{Existence of isoperimetric clusters}\label{section proof of existence} In this section we prove the existence statement in Theorem \ref{thm main}:

\begin{theorem}\label{thm existence}
  If $n,N\ge 2$, $s\in(0,1)$, and $m\in\R^N_+$, then there exist minimizers in the variational problem
  \begin{equation}
  \label{isoperimetric problem fractional existence 2}
  \g=\inf \big\{ P_s(\E): \E \mbox{ is an $N$-cluster in $\R^n$ with }m (\E) = m \big\}\,.
  \end{equation}
  Moreover, if $\E$ is a minimizer, then ${\rm diam}(\pa\E)<\infty$.
\end{theorem}

\begin{proof}
By explicit comparison with a cluster whose chambers are $N$ disjoint balls with suitable volumes we find that $\gamma<\infty$. Let us consider a minimizing sequence sequence $\{\E_k\}_{k\in \N}$ such that
$$ \lim_{k\to \infty}  P_s(\E_k) = \gamma \qquad m\,(\E_k)= m \qquad \forall k\in \N\,.$$
Let us set
$$m_{\min}= \min\{ m(h): 1\leq h\leq N\} \qquad m_{\max}= \max \{ m(h): 1\leq h\leq N\} $$
$$p_{\min}= \inf\{ P_s( \E_k(h)): 1\leq h\leq N, \; k\in \N\} \qquad m_{\max}= \sup\{ P_s( \E_k(h)): 1\leq h\leq N, \; k\in \N\}$$
so that $p_{\min} \geq P_s(B_1) m_{\min}^{(n-s)/n} /|B_1|^{(n-s)/n}>0$ by the isoperimetric inequality and $p_{\max} <\infty$ since $\gamma<\infty$.

\medskip

\noindent {\it Step one: first nucleation and construction of volume-fixing diffeomorphisms}. We apply the nucleation Lemma~\ref{lemma:nucleation} with $E= \E(h)$ and %$F= \emptyset$ and
$\e = \min \{m_{\min}, \frac{1-s}{\chi_1\chi_2} p_{\min} \}^{n/s}$ (where $\chi_1$ and $\chi_2$ depend only on $n$ and $s$ and are defined in \eqref{defn:xi12}). We obtain that there exist sequences $\{x_k(h)\}_{k\in \N}$ ($1\leq h\leq N$), such that for every $k\in \N$ and $1\leq h\leq N$
\begin{equation}
\label{eqn:nucl-primo}
\Big| \E_k(h) \cap B_1(x_k(h))\Big| \geq c,
\end{equation}
where $c$ depends only on $n,s, m_{\min}, p_{\max}$.
If we define $S$ by $\omega_n S^n = 2(m(1)+...+m(N))$, then at least half of the volume in $B_S(x_k(h))$ is occupied by the exterior chamber $\E_k(0)$, that is
$$ \big| \E_k(0) \cap B_S(x_k(h)) \big| \geq \frac{\omega_n S^n}{2}.$$
We apply %Lemma ... to obtain a system of interface points for $\E$ in  (see also Remark..., in order to apply the lemma in the open set $\cup_{h=1}^n B_S(x_k(h))$) and
Proposition~\ref{prop:vol-fix-k} to obtain the existence of positive constants $\eta<c_0/2$ and $C$ such that, up to extracting a not-relabeled subsequence in $k$, there exist $C^1$ functions
$$\Psi_k : ((-\eta, \eta)^N \cap V) \times \R^n \to \R^n$$
such that
for every $\bold a \in (-\eta,\eta)^{N+1}\cap V$ the map $\Psi_k(\bold a, \cdot): \R^n \to \R^n$ is a diffeomorphism with
\begin{equation}
\label{eqn:supp-psi-k}
\{ x\in \R^n : \Psi_k(\bold a,x) \neq x\} \subset  \cup_{h=1}^N B_S(x_k(h)) \cc \R^n
\end{equation}
\begin{equation}
\label{eqn:volexc-psi-k}
\Big| \Psi_k(\bold a, \E_k(h)) \cap \{ x\in \R^n : \Psi_k(\bold a,x) \neq x\}\Big| = \Big| \E_k(h) \cap \{ x\in \R^n : \Psi_k(\bold a,x) \neq x\}\Big| + \bold a (h);
\end{equation}
\begin{equation}
\label{eqn:psi-k-perimeter}
|P_s(\Psi_k(\bold a , F))- P_s(F)| \leq CP_s(F) \sum_{h=0}^N |\bold a (h)|,
\end{equation}
whenever $1\leq h\leq N$, $k\in \N$, and $F$ is a set of finite $s$-perimeter.

\medskip

\noindent {\it Step two: Fine nucleation of the cluster}. Let $\chi_1$ and $\chi_2$ be the constants in \eqref{defn:xi12}.  We prove that there exists a sequence of clusters $\{\E''_k\}_{k\in \N}$ such that for $k$ large enough
\begin{equation}
\label{eqn:pureminimizing}P_s(\E_k'') \leq P_s(\E)
\end{equation}
and
there are $r_0,\e_0>0$ and finitely many points $\{x_k(h,i)\}_{i=1,...,L(k,h)}$ with the property that
\begin{equation}
\label{eqn:e''1}\E''_k(h) \subseteq B_{r_0}(x_k(h)) \cup \bigcup_{i=1}^{L(k,h)} B_{r_0}(x_{k}(h,i)) \end{equation}
\begin{equation}
\label{eqn:e''2}\Big| \E''_{k}(h) \cap B_S(x_k(h))\Big| \geq \frac{c_0}{2} \qquad \mbox{for every } h=1,...,N
\end{equation}
\begin{equation}
\label{eqn:e''3}\sum_{j=1}^N\Big| \E''_{k}(j) \cap B_{r_0}(x_{k}(h,i))\Big| \geq \min\Big\{ \frac{c_0}{2} , \Big(\frac{\chi_1\,\e_0}{(1-s)\, p_{\max}}\Big)^{n/s} \Big\}\quad \mbox{for } i =1,...,L(k,h),\; h=1,...,N,
\end{equation}
$$L(k,h) \leq m_{\max} \Big(\frac{(1-s)\,p_{\max}}{\chi_1\,\e_0}\Big)^{n/s}.
$$
To this end, let $\e_0>0$ be such that
 \begin{equation}
  \label{eqn:eps0}
  \e_0 \leq \min \Big\{\eta, m_{\min}, \frac{1-s}{\chi_1\chi_2} p_{\min} \Big\}
  \end{equation}
 and, for every $k\in \N$ and $h=1,...,N$, let us apply Lemma~\ref{lemma:nucleation} to each chamber $\E_k(h)$ for finding finitely many points $\{ x_k(h,i)\}_{i=1,...,L(h,i)}$ with the property that
\begin{equation}
\label{eqn:almostcover}
\Big| \E_{k}(h) \setminus \bigcup_{i=1}^{L(k,h)} B_2(x_k(h,i))\Big| <\e_0,
\end{equation}
\begin{equation}
\label{eqn:unpocodimassa}
\Big| \E_{k}(h) \cap B_1(x_k(h,i))\Big| \geq \Big(\frac{\chi_1\,\e_0}{(1-s)\, p_{\max}}\Big)^{n/s} \qquad \forall i =1,...,L(h,i),
\end{equation}
$$L(k,h) \leq |\E_k(h)| \Big(\frac{(1-s)\,P_s(\E_{k}(h))}{\chi_1\,\e_0}\Big)^{n/s} \leq m_{\max} \Big(\frac{(1-s)\,p_{\max}}{\chi_1\,\e_0}\Big)^{n/s}.
$$
Next, for every $k\in \N$ we consider the closed set $F_k \subset \R^n$ given by
$$F_k:= \bigcup_{h=1}^N \Big( \overline B_S(x_k(h)) \cup \bigcup_{i=1}^{L(k,h)}\overline B_S(x_k(h,i)) \Big)
$$
and then we apply Lemma~\ref{lemma truncation} with $\tau = \e_0$ to each $\E_k$ and $F_k$. We set $C_1$ and $C_2$ as in \eqref{c1c2} depending only on $n$ and $s$, and we introduce the function $u_k =\dist(x,F_k)$ to find a sequence $\{r_k\}_{k\in \N} \subset [0, C_1\,\e_0^{1/n}]$ such that the clusters $\E_k'$ defined by
$$\E_k' (h) = \E_k(h) \cap \{ u_k \leq r_k\}, \qquad 1\leq h \leq N$$
satisfy
 \begin{equation}
   \label{tr1-applied}
 (1-s)\,P_s(\E_k' )\le (1-s)\,P_s(\E_k)-\frac{\dist(\E_k, \E_k')}{C_2\,\e_0^{s/n}}\,
 \end{equation}
(in particular $\lim_{k\to \infty} P_s(\E_k') = \gamma$). Finally, we set
$$\bold a_k(h) := |\E_k(h)| - |\E_k'(h)| = |\E_k(h) \cap \{ u_k >r_k\} | \qquad 1 \leq h \leq N, \qquad \bold a_k(0):= \sum_{h=1}^N \bold a_k (h).$$
By \eqref{eqn:almostcover} we have that $\bold a_k(h)\leq \e_0\leq \eta$, hence we can define
$$\E_k''(h) := \Psi_k(\bold a_k, \E'_k(h)) \qquad 1\leq h\leq N.$$
By \eqref{eqn:supp-psi-k} it follows that
$\{ x\in \R^n : \Psi_k(\bold a,x) \neq x\} \subset F_k \subset \{ u_k\leq r_k\}$, and thus for every $k \in \N$ and $h=1,.., N$ we have
$$
\Psi_k( \bold a_k, \E_k(h)) \cap \{ u_k \leq r_k\} = \Psi_k( \bold a_k, \E'_k(h)) \cap \{ u_k \leq r_k\} = \Psi_k( \bold a_k, \E'_k(h))\,,
$$
and
\begin{equation*}
\begin{split}
|\Psi_k(\bold a_k, \E'_k(h))| &= |\Psi_k( \bold a_k, \E_k(h)) \cap \{ u_k \leq r_k\}| = |\Psi_k( \bold a_k, \E_k(h))| - |\E_k(h) \cap \{u_k >r_k\} |
\\
&= |\Psi_k( \bold a_k, \E_k(h))| - \bold a_k = |\E_k(h)|,
\end{split}
\end{equation*}
that is, $m( \E''_k) = m(\E_k)= m$. We notice that \eqref{eqn:e''1} holds with $r_0= 2S+1+ C_1\,\e_0^{1/n}$. To prove \eqref{eqn:e''2}, we observe that
\begin{equation*}
\Big| \E''_{k}(h) \cap B_S(x_k(h))\Big| \geq\Big| \E_{k}(h) \cap B_S(x_k(h))\Big|- \bold a_k(h) \geq c_0 - \eta \geq \frac{c_0}{2}
\end{equation*}
To see that also \eqref{eqn:e''3} holds, given $h=1,...,N$ and $i=1,...,L(k,h)$, we consider two separate cases: if $B_1(x_k(h,i))$ intersects a ball $B_S(x_k(l))$ for some $l=1,...,N$, then $B_S(x_k(l)) \subseteq B_{r_0}(x_k(h,i)) $ and therefore
$$\sum_{j=1}^N\Big| \E''_{k}(j) \cap B_{r_0}(x_{k}(h,i))\Big| \geq | \E''_{k}(l) \cap B_{r_0}(x_k(l))\Big| \geq \frac{c_0}{2}$$
by \eqref{eqn:e''2}; if, instead, $B_1(x_k(h,i))$ does not intersect any of the balls $B_S(x_k(l))$, $l=1,...,N$, then \eqref{eqn:unpocodimassa} gives
\begin{equation*}
\begin{split}
\Big| \E''_{k}(h) \cap B_{r_0}(x_{k}(h,i))\Big|
&\geq \Big| \E''_{k}(h) \cap B_1(x_k(h,i))\Big|
\\
&=\Big| \E_{k}(h) \cap B_1(x_k(h,i))\Big|\geq
\Big(\frac{\chi_1\,\e_0}{(1-s)\, p_{\max}}\Big)^{n/s}
\end{split}
\end{equation*}
and thus \eqref{eqn:e''3} holds. Finally, we apply \eqref{eqn:psi-k-perimeter} to $\E_k'(h)$ and, using also \eqref{tr1-applied} and the equality $\sum_{h=0}^N |\bold a (h)| = \dist(\E_k, \E_k')$, we find that
\begin{equation*}
\begin{split}
P_s(\E_k'') &=P_s(\Psi_k(\bold a , \E_k')) \leq P_s(\E_k')+ |P_s(\Psi_k(\bold a , \E_k'))- P_s(\E_k'(h))|
\\
& \leq P_s(\E_k')+C P_s(\E_k') \sum_{h=0}^N |\bold a (h)|
\\
&\leq P_s(\E_k) -\frac{\dist(\E_k, \E_k')}{C_2(1-s)\,\e_0^{s/n}} +C P_s(\E_k') \sum_{h=0}^N |\bold a (h)|
\\
&\leq P_s(\E_k) -\frac{\dist(\E_k, \E_k')}{C_2(1-s)\,\e_0^{s/n}} +2 C\gamma  \dist(\E_k, \E_k')
\end{split}
\end{equation*}
which proves \eqref{eqn:pureminimizing} provided that we choose $\e_0$ small enough.

\medskip

\noindent {\it Step 3: boundedness of the new minimizing sequence, compactness and lower semicontinuity argument}. We conclude the proof. Lemma~\ref{lemma:boundedness} applied to the sequence of clusters $\E_k''$ with $R= r_0$ and $c=\min\{ c_0/2, [{\chi_1\,\e_0}{(1-s)^{-1}\, p^{-1}_{\max}}]^{n/s} \}$ implies that there exists $R_0>0$ such that, up to a subsequence not relabeled, $\E_k \subseteq B_{R_0}$ for every $k\in\N$. Therefore, each chamber $\E_k(h)$, $h=1,...,N$, converges in $L^1$ to a set $\E(h)$ which has volume $m(\E(h))= m(h)$ and perimeter $P_s(\E(h)) \leq \liminf_{k\to \infty} P_s(\E_k(h))$, by the lower semicontinuity of $P_s$ with respect to $L^1$ convergence of sets.
Hence
$$P_s(\E) = \sum_{h=0}^N P_s(\E(h)) \leq \sum_{h=0}^N \liminf_{k\to \infty} P_s(\E_k(h)) \leq \liminf_{k\to \infty} \sum_{h=0}^N P_s(\E_k(h)) = \gamma,$$
which proves that $\E$ is a minimizer for problem \eqref{isoperimetric problem fractional existence 2}.
\end{proof}

\section{Almost everywhere regularity}\label{section regularity}
We now address the regularity statements in Theorem \ref{thm main}, with the goal of proving the following statement:

\begin{theorem}\label{thm regularity}
If $n\ge2$ and $\E$ is an isoperimetric $N$-cluster in $\R^n$ (that is, $P_s(\E)\le P_s(\F)$ whenever $m(\F)=m(\E)$), then
there exists $\alpha \in (0,1)$ and a closed set $\S(\E)\subset\pa\E$ such that $\H^{n-2}(\S(\E))=0$ if $n\ge 3$, $\S(\E)$ is discrete if $n=2$, and for every $x\in \pa\E\setminus\S(\E)$ there exists $r_x>0$ such that $\partial \E \cap B_{r_x}(x)$ is a $C^{1,\alpha}$-hypersurface in $\R^n$.
In particular, $\partial \E$ is a locally $\H^{n-1}$-rectifiable set in $\R^n \setminus \S(\E)$ and it has Hausdorff dimension $n-1$.
% namely
%$$\H^s(\partial{\E}) =0 \qquad \mbox{for every }s >n-1.$$
\end{theorem}

The proof is divided in two parts. In section \ref{section lambdar0 mininimizing} we prove the $C^{1,\a}$-regularity of $\pa\E$ nearby points where $\E$ blows-up two complementary half-spaces. In section \ref{section extension}, following the approach of \cite{caffaroquesavin}, we estimate the dimensionality of the subset of $\pa\E$ where this blow-up property does not hold.

\subsection{Regular part of the boundary}\label{section lambdar0 mininimizing}
Given a $N$-cluster $\E$, $x\in\R^n$ and $r>0$ the {\it blow-up of $\E$ at $x$ at scale $r$} is the $N$-cluster $\E^{x,r}$ defined by
\[
\E^{x,r}(h)=\frac{\E(h)-x}r\,,\qquad h=1,...,N\,.
\]
The {\it regular set} ${\rm Reg}(\E)$ of $\E$ is the set of those $x\in\pa\E$ such that there exist an open half-space $H\subset\R^n$ and $h$, $k\in\{0,...,N\}$ such that, as $r\to 0^+$ and for every $j\ne h,k$,
\begin{equation}
  \label{regular point}
  \E^{x,r}(h)\to H\qquad\E^{x,r}(k)\to \R^n\setminus H\qquad \E^{x,r}(j)\to\emptyset\qquad\mbox{in $L^1_{{\rm loc}}(\R^n)$.}
\end{equation}
Our goal is proving that if $\E$ is an isoperimetric cluster, then ${\rm Reg}(E)$ is a $C^{1,\a}$-hypersurface in $\R^n$ which is relatively open in $\pa\E$.

We shall actually prove this fact for a larger class of clusters. Given an open set $A\subset\R^n$, $\Lambda\ge0$ and $r_0\in(0,\infty]$, we say that an $N$-cluster $\E$ is {\it $(\Lambda,r_0)$-minimizing in $A$} (it is tacitly understood that the word {\it minimizing} refers to $s$-perimeter)
\begin{equation}
\label{defn:alm-min}
P_s(\E;A)\le P_s(\F;A)+\frac{\Lambda}{1-s}\,\d(\E,\F)\,,
\end{equation}
whenever $\E\Delta\F\cc B_r(x)\cc A$, $r<r_0$. The use of perturbed minimality conditions such as \eqref{defn:alm-min} has been introduced in \cite{Almgren76} as a natural point of view for unifying regularity theorems. For example, as shown below, every isoperimetric cluster is $(\Lambda,r_0)$-minimizing in $\R^n$, but also every minimizer in the nonlocal partitioning problem
\[
\inf\Big\{P_s(\E;A)+\sum_{h=1}^N\int_{\E(h)}\,g_h(x)\,dx:\E(h)\setminus A=\E_0(h)\setminus A\quad h=1,...,N\Big\}
\]
(where $\E_0$ is a given $N$-cluster with $P_s(\E;A)<\infty$ and where $\{g_h\}_{h=1}^N\subset  L^\infty(A)$) is $(\Lambda,r_0)$-minimizing in $A'$ for every $A'\cc A$ (with $\Lambda$ and $r_0$ depending on the functions $g_h$ and on the distance between $A'$ and $A$). So minimizers in different variational problems satisfy analogous local almost-minimality conditions, which in turn imply several basic regularity properties.

\begin{proposition}\label{cor:almost-min}
  If $\E$ is an isoperimetric cluster in $\R^n$, then there exist constants $\Lambda\ge0$ and $r_0>0$ (depending on $\E$) such that $\E$ is $(\Lambda,r_0)$-minimizing in $\R^n$.
\end{proposition}

\begin{proof} Immediate from Proposition \ref{prop:vol-fix} and Lemma \ref{lemma:linkedchambers}.
%Let $\e_2$ and $C$ be the constants of the volume-fixing variations (Proposition~\ref{prop:vol-fix}). We can reduce to prove \eqref{defn:alm-min} under the further assumptions that $P_s(\F)\le P_s(\E)$ and that $d(\E, \F) \leq \e_2$.
%By the existence of a system of interface points of $\E$ (Lemma~\ref{lemma:linkedchambers}) and by the volume-fixing variations (Proposition~\ref{prop:vol-fix}), we find a diffeomorphism $\Psi:\R^n \to \R^n$ such that
%$|\Psi(\F(h))| = |\E(h)|$ for every $h = 1,..., N$
%and
%$$\sum_{h=0}^N \Big( P_s(\Psi(\F(h))) - P_s(\F(h)) \Big) \leq C d(\E, \F) \sum_{h=0}^N P_s(\F(h)) \leq C P_s(\E) d(\E, \F).$$
%%  This follows from the volume-fixing variation lemma.
\end{proof}

As explained at the beginning of the section, we aim to prove the following result.

\begin{theorem}\label{thm reg part 1}
  If $\E$ is a $(\Lambda,r_0)$-minimizing cluster in $\R^n$, then there exists $\a\in(0,1)$ such that ${\rm Reg}(\E)$ is a $C^{1,\a}$-hypersurface in $\R^n$ which is relatively open in $\pa\E$.
\end{theorem}

The next {\it infiltration lemma} (compare with \cite[Lemma 30.2]{maggiBOOK}) is a key step in proving Theorem \ref{thm reg part 1}.

\begin{lemma}\label{lemma:infiltration}
If $\E$ is a $(\Lambda, r_0)$-minimizing $N$-cluster in $\R^n$, then there exist positive constants $\s_0= \s_0(n,s, N)>0$, and $r_1 \leq r_0$ (depending on $n,s, \Lambda, r_0$) such that, if $x\in \R^n$, $r<r_1$, $h=0,...,N$ and
 $$|\E(h) \cap B_r(x)| \leq \s_0 r^n,$$
 then
 $$|\E(h) \cap B_{r/2}(x)| =0.$$
\end{lemma}
\begin{proof} We directly assume that $x=0$ and define an increasing function $u: (0,\infty) \to (0,\infty)$ by
\[
u(r) = |B_r \cap \E(h)|\qquad r>0\,,
\]
sot that $u'(r) = \H^{n-1}(\partial B_r \cap \E(h))$ for a.e. $r>0$. For every $r>0$, $i=0,...,N$, $i\neq h$, we consider the cluster obtained by giving part of the $h$-th chamber, namely $B_r \cap \E(h)$, to the $i$-th chamber
\begin{equation*}
\F_{r,i}(j) =
\begin{cases}
\E(h) \setminus B_r\qquad& \mbox{if }j=h\\
\E(i) \cup \big(\E(h)\cap B_r\big) \qquad& \mbox{if }j=i\\
\E(j)\qquad& \mbox{if }j\in \{0,...,N\} \setminus\{i,h\}.
\end{cases}
\end{equation*}
Since $\E$ is $(\Lambda,r_0)$-minimizing in $\R^n$ and since each $\F_{r,i}$ is an admissible competitor in \eqref{defn:alm-min}, we find that for every $r\leq r_1$, $i=0,...,N$, $i\neq h$,
\begin{equation}
\label{eqn:almostmin-appl}
\begin{split}
\frac{\Lambda}{1-s} u(r) &\geq P_s(\E)- P_s(\F_{r,i}) =  P_s(\E(i)) + P_s(\E(h))
- P_s(\F_{r,i}(i)) - P_s(\F_{r,i}(h)).
%\\
%&= I( \E(h) \cap B_r, \E(h)^c)+ I(\E(i) \setminus B_r,  \E(i) \cap B_r)- I(\E(h),  \E(i) \cap B_r)- I( \E(i) \cap B_r, \E(i) \setminus B_r)
%\\&=2\big(I(\E(h) \cap B_r, \E(i)) - I(\E(h) \cap B_r, \E(h) \cap B_r^c)\big).
\end{split}
\end{equation}
To estimate the right-hand side in \eqref{eqn:almostmin-appl} we compute
\begin{equation}
\label{eqn:diet1}
\begin{split}
&P_s(\F_{r,i}(i))- P_s(\E(i)) = I_s\big(\E(i) \cup \big(\E(h)\cap B_r\big), \E(i)^c \cap \big(\E(h)\cap B_r\big)^c \big)- I_s(\E(i), \E(i)^c)
\\&=I_s\big( \E(h)\cap B_r, \E(i)^c \cap \big(\E(h)\cap B_r\big)^c \big)+ I_s\big(\E(i) , \E(i)^c \cap \big(\E(h)\cap B_r\big)^c \big) - I_s(\E(i), \E(i)^c)
\\&= I_s\big(\E(h)\cap B_r, \E(i)^c \cap \big(\E(h)\cap B_r\big)^c \big)- I_s\big(\E(i), \E(h)\cap B_r \big)
\end{split}
\end{equation}
and
\begin{equation}
\label{eqn:diet2}
\begin{split}
&P_s(\F_{r,i}(h))- P_s(\E(h)) = I_s\big(\E(h)\cap B_r^c, \E(h)^c \cup B_r \big)- I_s(\E(h), \E(h)^c)
\\&= I_s\big(\E(h)\cap B_r^c, \E(h)^c \cup B_r \big)- I_s(\E(h)\cap B_r^c, \E(h)^c) - I_s(\E(h)\cap B_r, \E(h)^c)
\\&= I_s\big(\E(h)\cap B_r^c, \E(h) \cap B_r \big) - I_s(\E(h)\cap B_r, \E(h)^c)
\end{split}
\end{equation}
We notice that
\begin{equation}
\label{eqn:simm-diff-is}I_s(A,B) - I_s(A,C) = I_s(A,B\setminus C)- I_s(A,C\setminus B)
\end{equation}
for every triple of measurable sets $A,B,C \subseteq \R^d$. Hence the difference between the first term in the right-hand side of \eqref{eqn:diet1} and the second term in the right-hand side of \eqref{eqn:diet2} equals
$$ I_s\big(\E(h)\cap B_r, \E(i)^c \cap \big(\E(h)\cap B_r\big)^c \big)- I_s(\E(h)\cap B_r, \E(h)^c) =  I_s\big(\E(h)\cap B_r, \E(h) \cap B_r^c \big)- I_s\big(\E(i), \E(h)\cap B_r \big).
$$
We add the previous equations  \eqref{eqn:diet1} and  \eqref{eqn:diet2}, plugging them into \eqref{eqn:almostmin-appl}, and then we apply the last equality to find that, for every $r>0$, $i=0,...,N$, $i\neq h$,
\begin{equation*}
\begin{split}
\frac{\Lambda}{1-s} u(r) &\geq  I_s\big(\E(h)\cap B_r, \E(i)^c \cap \big(\E(h)\cap B_r\big)^c \big)- I_s\big(\E(i), \E(h)\cap B_r \big)
\\& \hspace{1em}+ I_s\big(\E(h)\cap B_r^c, \E(h) \cap B_r \big) - I_s(\E(h)\cap B_r, \E(h)^c)
\\&=2\Big(I_s(\E(h) \cap B_r, \E(i)) - I_s(\E(h) \cap B_r, \E(h) \cap B_r^c)\Big).
\end{split}
\end{equation*}
Averaging over $i \neq h$ we obtain that
\begin{equation}\label{eqn:almostmin-appl-new}
\begin{split}
\frac{\Lambda}{1-s} u(r) &\geq \frac{2}{N} \sum_{i\neq h}I_s(\E(h) \cap B_r, \E(i)) - 2I_s(\E(h) \cap B_r, \E(h) \cap B_r^c)
\\&= \frac 2 N I_s(\E(h) \cap B_r, \E(h)^c \cup B_r^c) - 2\Big(1+\frac 1 N \Big)I_s(\E(h) \cap B_r, \E(h) \cap B_r^c)
\end{split}
\end{equation}
where the last equality follows from the fact that
$$
\E(h)^c \cup B_r^c= \big(\E(h) \cap B_r^c \big)\cup \bigcup_{i\neq h} \E(i).
$$
 By the isoperimetric inequality, we have that
$$I_s(\E(h) \cap B_r, \E(h)^c \cup B_r^c) = P_s(\E(h) \cap B_r) \geq \frac{P_s(B)}{|B|^{(n-s)/n}} u(r)^{(n-s)/n}
$$
By the coarea formula and the fact that $u'(t) = \H^{n-1}(E \cap \partial B_t)$, we find
\begin{equation*}
\begin{split}
 I_s(\E(h) \cap B_r, \E(h) \cap B_r^c)
&\leq \int_{\E(h) \cap B_r} dx \int_{B(x,r-|x|)^c} \frac{dy}{|x-y|^{n+s}}
\\
&= \frac{P(B)}{s}  \int_{\E(h) \cap B_r} \frac{dx}{(r-|x|)^s} = \frac{P(B)}{s}  \int_{0}^r \frac{u'(t)}{(r-t)^s} \, dt .
\end{split}
\end{equation*}
Hence, from \eqref{eqn:almostmin-appl-new} we deduce that
\begin{equation}
\label{eqn:integral-ODE}
\frac{2P_s(B)}{N|B|^{(n-s)/n}} u(r)^{(n-s)/n} \leq 2\Big(1+\frac 1 N \Big) \frac{P(B)}{s}  \int_{0}^r \frac{u'(t)}{(r-t)^s} \, dt
+\frac{\Lambda}{1-s} u(r).
\end{equation}
Setting
$$r_1 = \min \Big\{ r_0, \Big( \frac{(1-s) P_s(B)}{N\Lambda |B|} \Big)^{1/s} \Big\}, \qquad c_0= \Big( \frac{s}{4(N+1)|B| 2^{n/s}} \frac{(1-s) P_s(B)}{P(B)}\Big)^{n/s},$$
%since $u(r) \leq |B| r^n $ for every $r>0$,
 we find that for every $r\leq r_1$
$$
\frac{\Lambda}{1-s} u(r) \leq \frac{\Lambda}{1-s} u(r)^{(n-s)/n}u(r)^{s/n}
\leq \frac{\Lambda}{1-s} u(r)^{(n-s)/n}|B|^{s/n} r_1^s
\leq \frac{P_s(B)}{N|B|^{(n-s)/n}} u(r)^{(n-s)/n}.
$$
Therefore, \eqref{eqn:integral-ODE} implies that
\begin{equation*}
%\label{eqn:integral-ODE}
 u(r)^{(n-s)/n} \leq 2(N+1) \frac{P(B)|B|^{(n-s)/n}}{sP_s(B)}  \int_{0}^r \frac{u'(t)}{(r-t)^s} \, dt.
\end{equation*}
By a De Giorgi-type iteration lemma (see \cite[Lemma 3.2 and Proof of Lemma 3.1]{F2M3}) this implies that $u(r) > c_0 |B| r^n$ for every $r\leq r_1$ and concludes the proof of the lemma.
\end{proof}

\begin{corollary}\label{cor:infiltration}
If $\E$ is a $(\Lambda, r_0)$-minimizing cluster in $\R^n$, then there exist positive constants $\s_0= \s_0(n,s, N)$ and $r_1 \leq r_0$ (depending on $n,s, \Lambda, r_0$) such that, if $x\in \R^n$, $r<r_1$, $S \subseteq \{0,...,N\}$ and
 $$ |\E(h) \cap B_r(x)| \leq \s_0 r^n \qquad \mbox{for every } h\in S,$$
 then
 $$|\E(h) \cap B_{2^{-N}r}(x)| =0\qquad \mbox{for every } h\in S.$$
\end{corollary}

\begin{proof}
Take the new $\s_0$ to be the one given by the previous lemma divided by $2^{nN}$. Then we can apply the Lemma~\ref{lemma:infiltration}
 iteratively to deduce that $k$ chambers in $S$ are not present in $B_{2^{-k}r}(x)$.
\end{proof}

\begin{corollary}
  \label{lemma perimeter volume estimate}
  If $\E$ is a $(\Lambda, r_0)$-minimizing cluster in $\R^n$, then there exist positive constants $r_1$ and $C_0$ (depending on $n$, $s$, $\Lambda$ and $r_0$) and $c_0, c_1\in (0,1)$ (depending on $n$ only), such that for every $r<r_1$, $x\in\R^n$ and $h=0,...,N$ one has
  \begin{eqnarray}\label{upper perimeter estimate}
  \sum_{h=0}^N P_s( \E(h) \cap B_r(x)) \leq C_0 r^{n-s}\,,
  \\
  \label{eqn:n-dim-density}
  c_0 \omega_n r^n \leq |\E(h) \cap B_r(x)| \leq c_1 \omega_n r^n\,.
  \end{eqnarray}
%  $$c_2 \omega_n r^{n-s} \leq P_s( \E(h) \cap B_r(x)),$$
  %{\color{blue} questa serve? per ora non e' dimostrata; con l'intersezione e' ovvia e segue dall'isoperimetrica, senno' serve l'isoperimetrica relativa e non so dove e' scritta}
\end{corollary}

\begin{proof}
%By Proposition~\ref{cor:almost-min}, we know that $\E$ is a $(\Lambda, r_0)$-minimizer for some $\Lambda, r_0>0$. Hence,
Clearly \eqref{eqn:n-dim-density} follows from Lemma \ref{lemma:infiltration}, so we focus on \eqref{upper perimeter estimate}. Comparing $\E$ to the cluster which is obtained from $\E$ by giving $B_r(x)$ to the exterior chamber in the $(\Lambda,r_0)$-minimality, we have that
$$
P_s(\E(0)) -  P_s(\E(0) \cup B_r(x)) +\sum_{h=1}^N \big(P_s(\E(h)) -  P_s(\E(h) \setminus B_r(x)) \big) \le\frac{\Lambda}{1-s}\,\d(\E,\F)\,,
$$
Since for every measurable sets $E,F$ we have that
\begin{equation}
\label{eqn:bell1}
\begin{split}
P_s(E\cap F) + P_s(E\setminus F) %&= I_s(E \cap F , E^c \cup F^c) + I_s(E \cap F^c, E^c \cup F) \\
& \leq I_s(E \cap F , E^c ) + I_s(E \cap F , F^c) + I_s(E \cap F^c, E^c )+ I_s(E \cap F^c, F)
\\&\leq P_s(E) +2 P_s( F)
\end{split}
\end{equation}
and similarly
\begin{equation}
\label{eqn:bell2}
P_s(E\cap F) + P_s(E \cup F) \leq P_s(E) + P_s( F),
\end{equation}
applying \eqref{eqn:bell1} to each chamber $E= \E(h)$ with $F = B_r(x)$ and applying \eqref{eqn:bell2} to $E= \E(0)$ with $F = B_r(x)$ we deduce that
$$
\sum_{h=0}^N P_s( \E(h) \cap B_r(x))\le (2N+1) P_s(B_r(x)) +\frac{\Lambda}{1-s}\,\d(\E,\F)\,,
$$
Since by scaling $P_s(B_r) = P_s(B_1) r^{n-s}$ and $\d(\E,\F) \leq \omega_n r^n \le \omega_n r_0^{s} r^{n-s}$, we have proved \eqref{upper perimeter estimate}.
\end{proof}

\begin{proof}
  [Proof of Theorem \ref{thm reg part 1}] {\it Step one}: We show that if $x\in{\rm Reg}(\E)$, then there exist $h=0,...,N$ and $s_x>0$ such that $x\in\pa\E(h)$ and $\E(h)$ is $(\Lambda',s_x)$-minimizing in $B_{s_x}(x)$. Indeed, by definition of ${\rm Reg}(\E)$, if $x\in{\rm Reg}(\E)$, then
  \[
  \lim_{r\to 0^+}\frac{|\E(h)\cap B_r(x)|}{|B_r(x)|}+\frac{|\E(k)\cap B_r(x)|}{|B_r(x)|}=1
  \]
  for some $h,k\in\{0,...,N\}$, $h\ne k$. Thus, by Corollary \ref{cor:infiltration}, there exists $r_x>0$ such that $\E(j)\cap B_{r_x}(x)=\emptyset$ if $j\ne h,k$. We now claim that if $s_x=\min\{r_x,r_0\}/2$, then there exists $\Lambda'\ge\Lambda$ such that
  \begin{equation}
    \label{atash}
      P_s(\E(h);B_{s_x}(x))\le   P_s(F;B_{s_x}(x))+\frac{\Lambda'}{1-s}\,|F\Delta\E(h)|
  \end{equation}
  whenever $F\Delta \E(h)\cc B_{s_x}(x)$. Indeed, given such a set $F$, let $\F$ be the cluster defined by
  \begin{eqnarray}
   \F(h) = F\,,\qquad \F(k) = (\E(k) \cap B_{s_x}(x)^c)  \cup( F^c \cap B_{s_x}(x))\,,\qquad\F(j) = \E(j)\quad\forall j\ne h,k\,.
  \end{eqnarray}
  Since $\E\Delta\F\cc B_{s_x}(x)$ and $s_x<r_0$ we have $P_s(\E)\le P_s(\F)+(1-s)^{-1}\Lambda\,\d(\E,\F)$, which in turn gives
  \[
  P_s(\E(h))+ P_s(\E(k))\le P_s(\F(h))+ P_s(\F(k))+\frac{2\Lambda}{1-s}\, \Big( |\E(h) \Delta \F(h) | + |\E(k) \Delta \F(k) | \Big)\,.
  \]
  We want to rewrite this condition in terms of $\E(h)$ and $\F(h)$ only: to this end, we set $R = \E(0) \cup \E(3) \cup ... \cup \E(N) $, and since $\E(h) = ( \F(h) \cup R)^c$ we thus find
\begin{equation}
\label{eqn:lamda-min-single}
0 \leq  P_s(\F(h))+ P_s(\F(h) \cup R) - P_s(\E(h))- P_s(\E(h) \cup R)+\frac{4 \Lambda}{1-s}\,  |\E(h) \Delta \F(h) |.
\end{equation}
We have that
\begin{equation*}
%\label{eqn:lamda-min-single1}
P_s(\F(h) \cup R) = %=  I_s(\F(h) \cup R, \F(h)^c  \cap R^c) =
 I_s(\F(h), \F(h)^c \cap R^c) + I_s( R, \F(h)^c  \cap R^c) = P_s(\F(h)) - 2 I_s(\F(h), R) + P_s(R)
\end{equation*}
and similarly
\begin{equation*}
%\label{eqn:lamda-min-single2}
P_s(\E(h) \cup R) = %=  I_s(\F(h) \cup R, \F(h)^c  \cap R^c) =
 I_s(\E(h), \E(h)^c \cap R^c) + I_s( R, \E(h)^c  \cap R^c) = P_s(\E(h)) - 2 I_s(\E(h), R) + P_s(R).
\end{equation*}
Plugging the last two equations in  \eqref{eqn:lamda-min-single} and dividing by $2$ we obtain
$$0 \leq   P_s(\F(h)) - P_s(\E(h)) + I_s(\E(h), R) -  I_s(\F(h), R)+\frac{2 \Lambda}{1-s}\,  |\E(h) \Delta \F(h) |.$$
Moreover, since $R$ and $\F(h) \setminus \E(h)$ are at distance $r_0/2$, by \eqref{eqn:simm-diff-is}
\begin{equation}\label{eqn:lamda-min-single3}
\begin{split}
I_s(\E(h), R) - I_s(\F(h), R) &= I_s(\E(h) \setminus \F(h), R) - I_s(\F(h) \setminus \E(h) , R) \leq  I_s(\E(h) \setminus \F(h), R) \\
&\leq \int_{\E(h) \setminus \F(h)} \int_{|x-y|>r_0/2} \frac{dx}{|x-y|^{n+s}}  \, dy = P(B_1)|\E(h) \setminus \F(h) |\int_{r_0/2}^\infty \frac{dr}{r^{1+s}}
%  \leq   |\E(h) \setminus \F(h) |
\\ &\leq \frac{2^sP(B_1)}{s r_0^s} |\E(h) \setminus \F(h) |.
\end{split}
\end{equation}
%Plugging  \eqref{eqn:lamda-min-single1} and \eqref{eqn:lamda-min-single2} in  \eqref{eqn:lamda-min-single} and estimating a term with \eqref{eqn:lamda-min-single3},
Hence we are left with
$$0 \leq   P_s(\F(h)) - P_s(\E(h)) +\frac{2 \Lambda}{1-s}\,  |\E(h) \Delta \F(h) | + \frac{2^sP(B_1)}{s r_0^s} |\E(h) \setminus \F(h) | ,
$$
which in turn proves the $(\Lambda',s_x)$-minimality of $\E(h)$  in $B_{s_x}(x)$ with $\Lambda' =2 \Lambda +(1-s) P(B_1)/ s  r_0^s$.

\medskip

\noindent {\it Step two}: Let $x\in{\rm Reg}(E)$ and let $s_x$ and $h$ as in step one. By \eqref{regular point} there exists an half-space $H$ such that $\E(h)^{x,r}\to H$ as $r\to 0^+$. By \eqref{eqn:n-dim-density}, given $\de>0$ and up to further decreasing the value of $s_x$ depending on $\de$, we may entail that
\[
B_{s_x}(x)\cap\pa\E(h)\subset \big\{y\in\R^n:\dist(y,\pa H)<\de\big\}\,.
\]
By the main result in \cite{caputoguillen} (see \cite{caffaroquesavin} for the case $\Lambda'=0$), if we take a suitable value of $\de$ (depending on $n$, $s$ and $\Lambda'$), then \eqref{atash} implies that $B_{s_x/2}(x)\cap\E(h)$ is contained in the epigraph of a $C^{1,\a}$ function defined of $(n-1)$-variables. This implies that $B_{s_x/2}(x)\cap\pa\E(h)\subset{\rm Reg}(\E)$, that $B_{s_x/2}(x)\cap\pa\E(h)=B_{s_x/2}(x)\cap\pa\E$, and that $B_{s_x/2}(x)\cap\pa\E(h)$ is a $C^{1,\a}$-hypersurface. The theorem is proved.
\end{proof}

\subsection{Blow-ups and monotonicity formula}\label{section extension} We now come to the problem of addressing the size of the singular set $\S(\E)=\pa\E\setminus{\rm Reg}(\E)$, consisting of those $x\in\pa\E$ such that $\E^{x,r}$ do not converge to a pair complementary half-spaces as $r\to 0^+$. The first step in this direction is showing that sequential blow-ups of minimizing clusters are conical (and still minimizing). In order to state the result, we introduce the following terminology.

A {\it conical $M$-cluster $\K$ in $\R^n$} is a $M$-cluster with the property that $r \K(i)  = \K(i)$ for every $r>0$ and $i=1,...,M$. We notice that, for conical clusters, being $(\Lambda,r_0)$-minimizing for some $\Lambda,r_0>0$ is equivalent to being $(0,\infty)$-minimizing. We thus simply speak of {\it minimizing conical clusters}. Finally, for any open set $A$ and for any pair of sets $E,F \subseteq \R^n$, we define the Hausdorff distance between $E$ and $F$ relative to $A$ as
  $$
  \hd_A(E,F) = \inf\{ \e>0: E \cap A \subseteq F_\e \mbox{ and } F \cap A \subseteq E_\e\},
  $$
  where $E_\e$ denotes the $\e$-enlargement of a set $E \subseteq \R^n$. We aim to prove the following theorem.

\begin{theorem}\label{thm cono tangente}
If $\E$ is a $(\Lambda,r_0)$-minimizing $N$-cluster in $\R^n$, $x\in \partial \E$, and $r_j \to 0^+$, then there exist a conical minimizing $M$-cluster $\K$ (with $M\leq N$) and an injective function $\sigma :\{ 0,..., M\} \to \{ 0,...,N \}$ such that, up to extracting subsequences
\begin{equation}
\label{eqn:blow-up-l1}\E_{x,r_j}(\sigma(i))\to \K(i) \qquad \mbox{in } L^1_{\loc}(\R^n)\qquad\forall i=0,...,M
\end{equation}
\begin{equation}
\label{eqn:hausd}
\hd_{B_R} (\partial\E_{x,r_j}(\sigma(h)) , \partial \K(h))\to 0  \qquad \forall R>0\,,i=0,...,M\,,
\end{equation}
as $j\to\infty$.
\end{theorem}

As usual, the key ingredient in proving Theorem \ref{thm cono tangente} is obtaining a monotonicity formula. Following \cite{caffaroquesavin}, this is obtained at the level of a degenerate Dirichlet energy associated to an extension problem, see Theorem \ref{thm:monot} below. The argument follow very closely \cite{caffaroquesavin}, so we limit ourselves to a quick review.

We start by introducing the extension problem and the Dirichlet form. Let $a=1-s$ and embed $\R^n$ into $\{z\ge0\}=\{X=(x,z)\in\R^{n+1}:z\ge 0\}$. We set
\[
U_R=\{X\in\R^{n+1}:|X|<R\}\,,\qquad A_+=A\cap\{z>0\},\qquad A_0=A\cap\{z=0\}\,,\qquad\forall A\subset\R^{n+1}\,.
\]
Given a measurable set $E$ we define $u_E: \{z\ge0\} \to \R$ by solving
\begin{equation*}
\begin{cases}
{\rm div}  (z^a \nabla u_E ) = 0 \qquad \mbox{in } \{z\ge0\}
\\
u_E = 1_E- 1_{E^c} \qquad \mbox{on } \{z=0\}\,.
\end{cases}
\end{equation*}
Notice that $u_E$ is obtained by convolution with the Poisson kernel,
$$
u_E(\cdot,z) = P(\cdot, z) \ast (1_E- 1_{E^c} ), \qquad P(x,z) = c(n,a) \frac{z^{1-a}}{( |x|^2+ z^2)^{\frac{n+1-a}{2}}}.
$$
If $E\subset\R^n$ is such that $P_s(E)<\infty$, then there exists a unique minimizer ${u}_E$ in
\[
\inf\Big\{\I_s(u)=\int_{\{z>0\}}z^a\,|\nabla u|^2:\tr(u)=1_E-1_{E^c}\Big\}\,,
\]
where $\tr$ denotes the trace operator from $\bigcap_{R>0}W^{1,1}(B_R\times(0,R))$ to $L^1_{loc}(\R^n)$, and one has
\[
\I_s({u}_E)=c\,P_s(E)\,,
\]
for some $c=c(n,s)$. The following lemma relates minimality for the nonlocal perimeter to minimality for the degenerate Dirichlet energy.

\begin{lemma}[Lemma 7.2 in \cite{caffaroquesavin}]\label{lemma:perim-ext} There exists a constant $c_0$ depending only on $n$ and $s$, and having the following property. If $E,F\subset\R^n$ are such that $P_s(E; B_1), P_s(F, B_1) <\infty$ and $E\Delta F\cc B_1$, then
\begin{equation}
\label{eqn:compet-extens}\inf_{\Omega, v} \int_{\Omega_+} z^a ( |\nabla v|^2 - |\nabla u_E|^2) \, dz
= c_{0} \big(P_s(F;B_1) - P_s(E;B_1) \big)
\end{equation}
where  $\Omega\subset\R^{n+1}$ is any bounded Lipschitz domain with $\Omega_0\subseteq B_1$ an $v \in W^{1,1}(\Omega)$ is such that ${\rm spt}(v- u_{E})\subset\Omega$ and $\tr v=1_{F} - 1_{F^c}$ on $\Om_0$.
\end{lemma}

\begin{corollary}\label{prop cluster min}
A cluster $\E$ is $(0,\infty)$-minimizing in $B_1$ if and only if for every $N$-cluster $\F$ with $\E\Delta\F\cc B_1$ the extensions $u_{\E(h)}$ of $\E(h)$ satisfy
\begin{equation}
\label{eqn:robin}
\sum_{h=0}^N \int_{\Omega_+} z^a |\nabla  u_{\E(h)}|^2\, dz \leq
 \sum_{h=0}^N \int_{\Omega_+} z^a |\nabla v_h|^2 \, dz
\end{equation}
for all bounded Lipschitz domains $\Omega\subset\R^{n+1}$ with $\Omega_0 \subseteq B_1$ and all functions $v_{h}$ such that ${\rm spt}(v_h- u_{\E(h)})\cc \Omega$ and $\tr v_h=1_{\F(h)}-1_{\F(h)^c}$ on $\Om_0$.
\end{corollary}

\begin{proof}
  Immediate from Lemma \ref{lemma:perim-ext}.
\end{proof}

We can now prove the following monotonicity formula.

\begin{theorem}[Monotonicity formula]\label{thm:monot}
  If $\E$ is a $(\Lambda,r_0)$-minimizing cluster with $0\in\pa\E$, then there exists $\Lambda'\ge0$ (of the form $\Lambda'=C(n,s)\Lambda$) such that
  \[
  \mbox{$\Phi_\E(r)+\Lambda'r^s$ is increasing on $(0,r_0)$}
  \]
  where we have set
  \begin{equation}\label{defn:phi}
    \Phi_\E(r)=\frac1{r^{n-s}}\,\sum_{h=0}^N \int_{U_r^+}z^a\,|\nabla {u}_{\E(h)}|^2\,.
  \end{equation}
  Moreover, if $r_0=\infty$ and $\Lambda=0$ then $\Phi_\E$ is constant if and only if $\E(h)$ is a cone with vertex at $0$ for every $h=0,...,N$.
\end{theorem}
%{\color{blue} qui concludo anche che $\E$ e' conical; infatti ottengo che l'estensione e' omogenea di grado 0, quindi sia l'insieme che il bordo hanno quella proprieta'...}

\begin{proof}
  The proof is again a simple adaptation of the argument in \cite{caffaroquesavin}. Given $\l\in(0,1)$ and $r>0$ let us set
  \[
  v_h={u}_{\E(h)}\,,\qquad v_h^\l(X)=\left\{
  \begin{array}{l l}
  v_h(X/\l)\,,&\mbox{if $|X|<\l\,r$}\,,
  \\
  v_h(rX/|X|)\,,&\mbox{if $\l\,r<|X|<r$}\,,
  \\
  v_h(X)\,,&\mbox{if $|X|>r$}\,,
  \end{array}
  \right .
  \]
  In this way
  \begin{eqnarray*}
    &&\int_{z>0}z^a|\nabla v_h^\l|^2-\int_{z>0}z^a|\nabla v_h|^2
    \\
    &=&
    (\l^{n-s}-1)\int_{U_r^+}|\nabla v_h|^2+\int_{\l r}^rds\int_{(\pa U_s)^+}z^a\,|\nabla_\tau v^\l_h|^2\,,
  \end{eqnarray*}
  where $\nabla_\tau v(X)=\nabla v(X)-|X|^{-2}(\nabla v(X)\cdot X)X$. We now notice that
  \[
  \tr v_h^\l=1_{\F^\l(h)}-1_{\F^\l(h)^c}\,,
  \]
  where
  \begin{eqnarray*}
    \F^\l(h)\setminus B_r&=&\E(h)\setminus B_r\,,
    \\
    \F^\l(h)\cap B_{\l r}&=&\l\E(h)\cap B_{\l r}\,,
    \\
    \F^\l(h)\cap (B_r\setminus B_{\l r})&=&(\R_+(\E(h)\cap\pa B_r))\cap (B_r\setminus B_{\l r})\,.
  \end{eqnarray*}
  Since $\F^\l\Delta\E\subset B_r$ if $r<r_0$ then we find
  \[
  P_s(\E)\le P_s(\F^\l)+\frac{\Lambda}{1-s}\,\d(\E,\F^\l)\,,
  \]
  which  by \eqref{eqn:compet-extens} takes the form
  \[
  \sum_{h=0}^N\int_{U_r^+}z^2|\nabla v_h|^2\le \int_{U_r+}z^2|\nabla v_h^\l|^2+\frac{2\Lambda}{(1-s)\,c_0}\,|m(\E)-m(\F^\l)|\,,
  \]
  that is
  \[
  (1-\l^{n-s})\sum_{h=0}^N\int_{U_r^+}|\nabla v_h|^2
  \le
  \int_{\l r}^rds\sum_{h=0}^N\int_{(\pa U_s)^+}z^a\,|\nabla_\tau^\l v_h|^2+\frac{2\Lambda}{(1-s)\,c_0}\,|m(\E)-m(\F^\l)|\,.
  \]
  Now
  \[
  |m(\E)-m(\F^\l)|=\sum_{h=1}^N\,||\E(h)|-|\F^\l(h)||
  \le (1-\l^n)\sum_{h=1}^N\,|\E(h)\cap B_r|+|B_r\setminus B_{r\l}|\le C\,r^n\,(1-\l^n)\,.
  \]
  In this way since $\l\in(0,1)$
  \[
  \frac{1-\l^{n-s}}{1-\l}\,\sum_{h=0}^N\int_{U_r^+}|\nabla v_h|^2
  \le
  \frac1{1-\l}\int_{\l r}^rds\sum_{h=0}^N\int_{(\pa U_s)^+}z^a\,|\nabla_\tau v_h|^2+C(n,s)\,\Lambda\,r^n\,\frac{1-\l^n}{1-\l}\,.
  \]
  We notice that $|\nabla v^\l_h(X)| = |\nabla v_h(rX/|X|)|$ for every $X$ with $\l\,r<|X|<r$ and we let $\l\to 1^-$ to find,
  \[
  (n-s)\sum_{h=0}^N\int_{U_r^+}|\nabla v_h|^2\le r\,\sum_{h=0}^N\int_{(\pa U_r)^+}z^a\,|\nabla_\tau v_h|^2+C(n,s)\,\Lambda\,r^n\,.
  \]
  Therefore
  \begin{eqnarray*}
  r^{2(n-s)}\,\Phi_\E'(r)&=&r^{n-s-1}\bigg\{r\sum_{h=0}^N \int_{(\pa U_r)^+}z^a\,|\nabla v_h|^2-(n-s)\,\sum_{h=0}^N\int_{U_r^+}|\nabla v_h|^2\bigg\}
  \\
  &\ge&r^{n-s-1}\bigg\{r\sum_{h=0}^N \int_{(\pa U_r)^+}z^a\,|\nabla v_h\cdot\hat{X}|^2-C(n,s)\Lambda\,r^n\bigg\}\,,
  \end{eqnarray*}
  where $\hat{X}=X/|X|$. Rearranging terms we find
  \[
  \Big(\Phi_\E(r)+\frac{C(n,s)\Lambda}s\,r^s\Big)'\ge \frac1{r^{n-s}}\sum_{h=0}^N \int_{(\pa U_r)^+}z^a\,|\nabla v_h\cdot\hat{X}|^2\,,
  \]
  for every $r<r_0$. This proves that $\Phi_\E(r)+{\Lambda}\,r^s$ is increasing on $(0,r_0)$. %, with ${\Lambda}=0$ if and only if $\Lambda=0$.
  Assume now that $r_0=\infty$ and $\Lambda=0$. In this case $\Phi_\E$ is increasing on $(0,\infty)$, and $\Phi_\E$ is constant on $(0,\infty)$ if and only if $\nabla v_h$ is homogeneous of degree $0$ for every $h=0,...,N$, that is if and only if $\E(h)$ is a cone with vertex at the origin for every $h=0,...,N$.
  \end{proof}

\begin{proof}[Proof of Theorem \ref{thm cono tangente}] Without loss of generality let us assume that $x=0$, so that $\E^{0,r}(h)=r^{-1}\E(h)$.
By the upper perimeter estimate of Lemma~\ref{lemma perimeter volume estimate}, for every $R,r>0$ and $h=0,...,N$ we have
$$
P_s\big((r^{-1}\E(h)) \cap B_R\big) = r^{s-n} P_s(\E(h) \cap B_{R\,r}(x)) \leq C_0 R^{n-s}\,.
$$
In particular for every $h=0,...,N$ there exists $\F(h)\subset\R^n$ such that, up to extracting subsequences, $r_j^{-1}\E(h)\to\F(h)$ in $L^1_{{\rm loc}}(\R^n)$ as $j\to\infty$. Define $M\le N$ so that there are exactly $M+1$ indexes $h=0,...,N$ such that $|\F(h)|>0$. Then we can find a injective function $\s:\{0,...,M\}\to\{0,...,N\}$ such that, setting $\K(i)=\F(\s(i))$, we have $r_j^{-1}\,\E(\s(i))\to\K(i)$ in $L^1_{{\rm loc}}(\R^n)$ as $j\to\infty$. This proves \eqref{eqn:blow-up-l1}, which in turn implies \eqref{eqn:hausd} thanks to the volume density estimates in Lemma \ref{lemma perimeter volume estimate}. Since $r_j\,\E$ is $(\Lambda\,r_j,r_j/r_0)$-minimizing in $\R^n$, by a simple variant of \cite[Theorem 3.3]{caffaroquesavin} we see that $\K$ is $(0,\infty)$-minimizing in $\R^n$. Moreover, by scaling
\[
\Phi_{\E}(r_j\,r)=\Phi_{r_j^{-1}\E}(r)\qquad\forall r>0
\]
so that
\[
\lim_{j\to\infty}\Phi_{r_j^{-1}\E}(r)=\Phi_{\E}(0^+)\,,\qquad\forall r>0\,.
\]
At the same time, by arguing as in \cite[Proposition 9.1]{caffaroquesavin}, we get
\[
\lim_{j\to\infty}\Phi_{r_j^{-1}\E}(r)=\Phi_\K(r)\,,\qquad\forall r>0\,.
\]
In conclusion, $\Phi_\K(r)$ is constant over $r>0$, and since $\K$ is $(0,\infty)$-minimizing in $\R^n$ we can exploit Theorem \ref{thm:monot} to deduce that $\K$ is conical.
\end{proof}

We conclude this section with a last result that can be proved with the aid of the extension problem and that it is useful in the dimension reduction argument (see next section).

\begin{proposition}\label{prop:cono-x-R}
A cluster $\E$ is $(0,\infty)$-minimizing  in $\R^n$ if and only if $\E \times \R$ is $(0,\infty)$-minimizing  in $\R^{n+1}$. Here, by definition, $(\E\times\R)(h)=\E(h)\times\R$ for every $h=1,...,N$.
\end{proposition}

\begin{proof}
  This is an immediate adaptation of \cite[Theorem 10.1]{caffaroquesavin}.
\end{proof}

\subsection{Dimension reduction argument} Given Theorem \ref{thm cono tangente} and Proposition \ref{prop:cono-x-R} we can exploit the standard dimension reduction argument of Federer to give estimates on the Hausdorff dimension of $\S(\E)$.

\begin{theorem}[Dimension reduction]\label{thm:dim-red} If $\K$ is a minimizing conical $M$-cluster in $\R^n$, $x_0=e_n \in \partial \K$ and $\l_k\to\infty$ as $k\to\infty$, then there exists a minimizing conical cluster $\K'$ in $\R^{n-1}$ such that, up to extracting subsequences,
\[
\lambda_k (\K-x_0)\to \K' \times \R\qquad\mbox{in $L^1_{\rm loc} (\R^n)$}
\]
as $k\to\infty$.
\end{theorem}

\begin{proof}
By Theorem \ref{thm cono tangente} there exists a conical minimizing $M$-cluster $\overline\K$ such that, up to extracting subsequences, $\lambda_k (\K-e_n)\to\overline{\K}$ in $L^1_{{\rm loc}}(\R^n)$. We want to prove that $\overline{\K}=\K' \times \R$ for some conical cluster $\K'$ in $\R^{n-1}$, and the fact that $\K'$ is minimizing will then follow by Proposition~\ref{prop:cono-x-R}.  Since $\partial \overline{\K}$ is a closed set of measure $0$ thanks to the density estimates, it is enough to prove that the interior of each chamber is constant in the $x_n$-direction, namely that for every chamber $\overline{\K}(h)$ and for every ball $B_\e(x) \subseteq \overline{\K}(h)$ we have
\begin{equation}
\label{eqn:K-splits}
B_\e(x)+\R e_n \subseteq \overline\K(h).
\end{equation}
To prove this claim, we notice that the cone with vertex in $-\lambda_k e_n$ generated by $B_\e(x)$ converges locally to $B_\e(x)+\R e_n $.
Moreover, setting $\K_k = \lambda_k (\K-e_n)$, we have that $B_\e(x) \cap {\K_k}(h)$ converges to $B_\e(x) \cap \overline{\K}(h) = B_\e(x)$.
As a consequence, the  difference between the indicator of the cones with vertex in $-\lambda_k e_n$ generated by $B_\e(x)$, and by $B_\e(x) \cap {\K_k}(h)$ respectively, converges in $L^1_{\rm loc}(\R^n)$
 to $0$. Putting together these facts, we deduce that the cone with vertex in $-\lambda_k e_n$ generated by $B_\e(x) \cap {\K_k}(h)$ (which is contained in $\K_k(h)$ because by assumption $\K_k(h)$ is a cone with vertex $-\lambda_k e_n$)
 converges in $L^1_{\rm loc}(\R^n)$ to $B_\e(x)+\R e_n $. By the convergence of $\K_k(h)$ to $\overline\K(h)$, we find that \eqref{eqn:K-splits} holds.
\end{proof}

\begin{theorem}[Dimension of the singular set]\label{thm:sing-set-dim} If $\E$ is a $(\Lambda,r_0)$-minimizing $N$-cluster in $\R^n$, then the singular set $\Sigma(\E)$ is a closed set of  Hausdorff dimension at most $n-2$, that is,
$$
\H^\ell(\S(\E)) =0 \qquad \forall \ell >n-2.
$$
As a consequence, $\partial \E$ has Hausdorff dimension $n-1$, namely
$$
\H^\ell(\partial{\E}) =0 \qquad \forall\ell>n-1.
$$
\end{theorem}
\begin{proof}
From Theorem~\ref{thm:dim-red} and Proposition~\ref{prop:cono-x-R} it follows that the singular set of any minimizing cluster $\E$ has Hausdorff dimension $n-2$. This is a classical argument, which can be repeated {\it verbatim} from \cite[Proof of Theorem 10.4]{caffaroquesavin}: first, one proves that $\H^\ell(\S(\E)) = 0$ for any $\ell$ such that $\H^\ell(\S(\K)) = 0$ for every conical minimizing cluster $\K$; next, one shows that $\H^\ell(\S(\K)) = 0$ for every conical minimizing cluster $\K \subseteq \R^n$, then  $\H^{\ell+1}(\Sigma_{\tilde \K}) = 0$ for every conical minimizing cluster $\tilde \K \subseteq \R^{n+1}$. In proving both claims one uses a compactness argument to say that for every $x\in \S(\E)$ there exists $\delta(x)>0$ such that for any $\delta \leq \delta(x)$ and any set $D\subseteq \S(\E) \cap B_\de(x)$ there exists a covering of $D$ with balls $B_{r_i}(x_i)$ such that $x_i\in D$ and $\sum r_i^\ell \leq \delta^\ell /2$. Finally, since $\partial \E$ is a $C^1$-hypersurface in a neighborhood of each $x\in{{\rm Reg}}(\E)$, we conclude that $\partial \E$ has Hausdorff dimension $n-1$.
\end{proof}

In the planar case $n=2$ we can say more by exploiting the fact, proved in \cite{savinvaldinoci}, that every conical minimizing $2$-cluster in $\R^2$ is given by two complementary half-spaces. By definition of ${\rm Reg}(\E)$, this fact implies that if $x\in\S(\E)$ for a $(\Lambda,r_0)$-minimizing cluster in $\R^2$ and $\K$ is a conical minimizing $M$-cluster arising as a blow-up limit of $\E$ at $x$, then $M\ge 2$ (that is, $\K$ has at least three non-trivial conical sectors). With this remark in mind we can prove the following fact.

\begin{proposition}\label{prop:dim2} The singular set $\S(\E)$ of a $(\Lambda,r_0)$-minimizing cluster $\E$ in $\R^2$ is locally discrete.
\end{proposition}

\begin{proof} Assume by contradiction that there exists a sequence $\{x_k \} _{k \in \N} \subseteq \S(\E)$ such that $x_k$ converges to $x_0 \in \S(\E)$ as $k \to \infty$. Set
\[
\lambda_k = |x_k-x_0|^{-1}\qquad \E_k=\E^{x_0,\l_k^{-1}}\,,
\]
and assume up to rotations that
\[
\frac{x_k-x_0}{|x_k-x_0|}=v\in S^1\qquad\forall k\in\N\,.
\]
In this way, $\E_k$ is $(\Lambda/\l_k\,,r_0\,\l_k)$-minimizing in $\R^2$ with $0,v\in\S(\E_k)$ for every $k\in\N$. By Theorem~\ref{thm cono tangente}, up to extracting subsequences, $\lambda_k (\E-x_0)\to\overline{\K}$ in $L^1_{{\rm loc}}(\R^2)$ with $0,v\in\S(\overline{\K})$. The fact that $v\in\S(\overline{\K})$ is based on the fact that, as notice above, by \cite{savinvaldinoci} $x_k\in\S(\E_k)$ implies (up to extracting a subsequence in $k$ and up to reordering the chambers of $\E$) that $|\E_k \cap B_r(v)|> 0$ for $h=1,2,3$ and for every $r>0$ and $k\in \N$.
Moreover, by the density estimates of Lemma~\ref{lemma perimeter volume estimate}
 (note that they are uniform with respect to $k$, since they are applied to a blow-up of a single cluster and so they hold at every scale less than $r_0(\E)$ as $k$ increases)
$$|\E_k\cap B_r(v)|  \geq c r^n$$
for every $h=1,2,3$, for every $r$ and for every $k$ large enough (depending on $r$). Thus there are at least three chambers of $\K$ which have positive volume nearby $v$, so that $v\not\in{\rm Reg}(\overline{\K})$. By Theorem~\ref{thm:dim-red} any blow-up of $\overline{\K}$ at $v$ has the form $\K' \times \R$ for some conical cluster $\K'$ in $\R$. Since the only nontrivial conical cluster in $\R$ is the half-line, we find that $\K' \times \R$ is actually an half-space. Hence, by Theorem \ref{thm reg part 1}, $v\in{\rm Reg}(\overline{\K})$. We have obtained a contradiction and the proof is complete.
\end{proof}

\begin{proof}
  [Proof of Theorem \ref{thm regularity}] Combine Theorem \ref{thm reg part 1}, Theorem \ref{thm:sing-set-dim} and Proposition \ref{prop:dim2}.
\end{proof}

\begin{proof}
  [Proof of Theorem \ref{thm main}] Combine Theorem \ref{thm existence} and Theorem \ref{thm regularity}.
\end{proof}

\bibliography{references}
\bibliographystyle{is-alpha}

\end{document}